\documentclass[twoside]{IEEEtran}
\usepackage[english]{babel}
\usepackage{microtype}
\usepackage{setspace}
\usepackage{blindtext}
\usepackage{siunitx}
\ifCLASSOPTIONcompsoc
  \usepackage[nocompress]{cite}
\else
  \usepackage{cite}
\fi

\usepackage{graphicx}
\usepackage{subfloat}
\usepackage{fancyhdr} 
\usepackage{color}
\usepackage{epsfig}
\graphicspath{{Images/}}
\usepackage{tikz}
\usepackage{tikzscale}
\usepackage{scalerel}
\usetikzlibrary{arrows}
\usetikzlibrary{plotmarks}
\usetikzlibrary{svg.path}
\usetikzlibrary{shapes.multipart}
\usepackage{pgfplots}
\pgfplotsset{compat=newest}
\pgfplotsset{every axis/.append style={
		label style={font=\Large},
		tick label style={font=\large}  
}}
\tikzstyle{int}=[draw, fill=black!10, minimum size=5em,thick]
\tikzstyle{init} = [pin edge={to-,thick,black}]

\usepackage{array}
\usepackage{stfloats}

\usepackage{amsthm,amsmath,amssymb,amsfonts}
\usepackage{relsize}
\usepackage{nicefrac}
\usepackage{bbm}
\usepackage{mathtools}
\usepackage[mathscr]{eucal}

\newcommand{\versus}{\emph{vs.}\xspace}

\usepackage{enumitem}
\usepackage[ruled]{algorithm}
\usepackage{algpseudocode}
\usepackage{multirow}
\makeatletter
\gdef\Shortstack{\@ifnextchar[\@Shortstack{\@Shortstack[c]}}
\gdef\@Shortstack[#1]#2{%
	\leavevmode
	\vbox\bgroup
	\baselineskip-\p@\lineskip 3\p@
	\let\mb@l\hss\let\mb@r\hss
	\expandafter\let\csname mb@#1\endcsname\relax
	\let\\\@stackcr\setlength{\baselineskip}{#2}%
	\@ishortstack}
\makeatother
\usepackage{slashbox}
\usepackage{booktabs}
\usepackage{tabularx}

\usepackage{float}
\usepackage[absolute]{textpos}

\usepackage{url}
\usepackage{hyperref}
\usepackage[capitalize]{cleveref}

\newcommand\orcidicon[1]{\href{https://orcid.org/#1}{\includegraphics[scale=0.05]{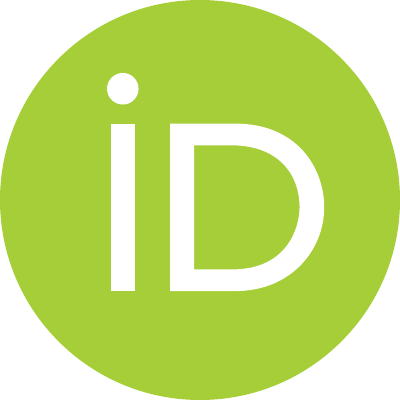}}}

\newcommand{\tradeoff}{communication-computation trade-off\xspace}
\newcommand{\problemName}{{\color{black}Optimal Estimation in Processing Network}} 

\newcommand{\Real}[1]{ { {\mathbb R}^{#1} } }
\newcommand{\Realp}[1]{ { {\mathbb R}^{#1}_+ } }

\DeclareMathOperator*{\argmin}{arg\,min}

\newcommand{\timeUp}{\mathscr{P}}
\newcommand{\update}{\mathscr{U}}
\newcommand{\iteration}{\mathscr{I}}

\theoremstyle{plain}
\newtheorem{thm}{Theorem}
\newtheorem{cor}{Corollary}
\newtheorem{prop}{Proposition}
\theoremstyle{definition}
\newtheorem{definition}{Definition}
\newtheorem{prob}{Problem}
\newtheorem{ass}{Assumption}
\theoremstyle{remark}
\newtheorem{rem}{Remark}
\newtheorem{ex}{Example}

\newcommand{\delayCont}{\tau}
\newcommand{\delayCommCont}{\delayCont_{c}(\delayCont)}
\newcommand{\delayFusCont}{\delayCont_{f}(\delayCont)}
\newcommand{\psteady}{p_{\infty|\infty-\delayCont_{\textit{tot}}}}

\newcommand{\delay}{\delta}
\newcommand{\delayDisc}[1]{\tau_{#1}}
\newcommand{\delaySetDisc}[2]{\mathcal{T}_{#1}^{#2}}
\newcommand{\delayCommDisc}{\delayDisc{c,i}}
\newcommand{\delayFusDisc}{\delayDisc{f,i}}
\newcommand{\delayFusDiscTot}{\delayDisc{\textit{f,tot}}}
\newcommand{\delayDiscTot}[1]{\tilde{\delayDisc{}}_{#1}}

\newcommand{\Psteady}{P_{\infty|\infty-\delayDisc{\textit{tot}}}}
\newcommand{\tr}[1]{\mbox{tr}\left(#1\right)}
\newcommand{\gauss}{\mathcal{N}}
\newcommand{\sensSet}{\mathcal{V}}
\newcommand{\sensNum}{V}
\newcommand{\cardS}{s\xspace}

\algdef{SE}[DOWHILE]{Do}{doWhile}{\algorithmicrepeat}[1]{\algorithmicuntil\ #1}%

\renewcommand{\algorithmiccomment}[1]{\bgroup\hfill//~#1\egroup}

\addto\captionsenglish{\renewcommand{\figurename}{Fig.}}
\addto\captionsenglish{\renewcommand{\tablename}{TABLE}}

\newcommand{\blue}[1]{{\color{blue}#1}}
\newcommand{\revision}[1]{{\color{black}#1}}

\newcommand{\linkToPdf}[1]{\href{#1}{\blue{(pdf)}}}
\newcommand{\linkToPpt}[1]{\href{#1}{\blue{(ppt)}}}
\newcommand{\linkToCode}[1]{\href{#1}{\blue{(code)}}}
\newcommand{\linkToWeb}[1]{\href{#1}{\blue{(web)}}}
\newcommand{\linkToVideo}[1]{\href{#1}{\blue{(video)}}}
\newcommand{\linkToMedia}[1]{\href{#1}{\blue{(media)}}}
\newcommand{\award}[1]{\xspace} 
\newcommand{\eg}{\emph{e.g.,}\xspace}
\newcommand{\ie}{\emph{i.e.,}\xspace}
\addto\extrasenglish{}
\addto\extrasenglish{}
\addto\extrasenglish{}

\begin{document}

	\title{Computation-Communication Trade-offs and Sensor Selection in Real-time Estimation for Processing Networks}

	\author{Luca~Ballotta~{\orcidicon{0000-0002-6521-7142}}, %
	Luca~Schenato~{\orcidicon{0000-0003-2544-2553}},~\IEEEmembership{Fellow,~IEEE}, %
	Luca~Carlone~{\orcidicon{0000-0003-1884-5397}},~\IEEEmembership{Senior~Member,~IEEE} %
	\IEEEcompsocitemizethanks{\IEEEcompsocthanksitem L. Ballotta and L. Schenato are with the Department
		of Information Engineering, University of Padova, Padova, 35131, Italy.\protect\\
		E-mail: {\{ballotta, schenato\}@dei.unipd.it}
	\IEEEcompsocthanksitem L. Carlone is with the Laboratory for Information \& Decision Systems,
		Massachusetts Institute of Technology, Boston, 02139, USA.\protect\\
		E-mail: lcarlone@mit.edu}%
	}

	\IEEEtitleabstractindextext{%

\begin{abstract}     
Recent advances on hardware accelerators and edge computing are enabling substantial processing to be performed at each node 
(e.g., robots, sensors) of a networked system. 
Local processing typically enables data compression and may help mitigate measurement noise, 
but it is still usually slower compared to a central computer
(\ie it entails a larger \emph{computational} delay).
Moreover, while nodes can process the data in parallel, the \revision{computation} at the central computer 
is sequential in nature.
On the other hand, if a node decides to 
send raw data to a central computer for processing, it incurs a \emph{communication} delay. 
This leads to a fundamental \emph{\tradeoff}, where each node has to decide on the optimal amount of local preprocessing 
in order to maximize the network performance.
Here we consider the case where the network is in charge of estimating the state of a dynamical system 
and provide three key contributions.
First, we provide a rigorous problem formulation for optimal real-time estimation in \emph{processing networks},
in the presence of communication and computation delays.
Second, we develop analytical results for the case of a homogeneous network (where all sensors have the same computation) that monitors a 
continuous-time scalar linear system. In particular, we show how to compute the optimal amount of local preprocessing 
to minimize the estimation error and prove that sending raw data is in general suboptimal in the
presence of communication delays.
Third, we consider the realistic case of a heterogeneous network that monitors a discrete-time multi-variate linear system 
and  provide practical algorithms (i) to decide on a suitable preprocessing at each node, 
and (ii) to select a sensor subset when computational constraints make using all sensors suboptimal.
Numerical simulations show that selecting the sensors is crucial: the more may not be the merrier.
Moreover, we show that if the nodes apply the preprocessing policy suggested by our algorithms, they can largely improve the network estimation performance.
\end{abstract}

}

	\maketitle
	\begin{textblock}{10}(3,.05)
		\footnotesize
		\centering
		\setstretch{1}
		This paper has been accepted for publication in IEEE Transactions on Network Science and Engineering, 2020.\\
		Please cite the paper as: L. Ballotta, L. Schenato, and L. Carlone,\\
		``Computation-Communication Trade-offs and Sensor Selection in Real-time Estimation for Processing Networks'',\\
		IEEE Trans. on Network Science and Engineering, 2020	
	\end{textblock}
	\begin{textblock}{10}(3,15.2)
		\footnotesize
		\centering
		\setstretch{1}
		\textcopyright 2020 IEEE.  
		Personal use of this material is permitted.  
		Permission from IEEE must be obtained for all other uses, in any current or future media, including reprinting/republishing this material for advertising or promotional purposes, creating new collective works, for resale or redistribution to servers or lists, or reuse of any copyrighted component of this work in other works.
	\end{textblock}

	\IEEEdisplaynontitleabstractindextext

	\ifCLASSOPTIONpeerreview
		\begin{center} 
			\bfseries EDICS Category: 3-BBND
		\end{center}
	\fi
	
	\IEEEpeerreviewmaketitle


{\section{Introduction}}

\IEEEPARstart{I}{n} the last decade, networked control systems and related domains, such as \emph{Internet of Things} (IoT) 
and \emph{Cyber-Physical Systems} (CPS), have witnessed an unprecedented growth.
The increasing interest towards these systems stems from the broad range of applications they enable:
air-pollution monitoring for city air quality~\cite{8405565}, smart grids~\cite{Pasqualetti11cdc-safeCPS}, robot swarms for target tracking~\cite{4937860}, and autonomous car networks~\cite{Shalev-Shwartz17arxiv-safeDriving}, to mention a few.
On the one hand, the deployment of increasingly powerful communication protocols, such as 5G, carries the promise of further enhancing
the communication capabilities and scale of future networked systems.
On the other hand, advances in electronics, such as embedded GPU-CPU systems~\cite{TX2website} and dedicated hardware~\cite{Suleiman18jssc-navion} for embedded systems, make \textit{edge computing} an appealing and powerful strategy, 
where the nodes in the network can locally preprocess the acquired data, 
saving communication and possibly reducing the workload of a central station collecting and 
post-processing the data.

\begin{figure}
	\centering
	\includegraphics[width=\linewidth]{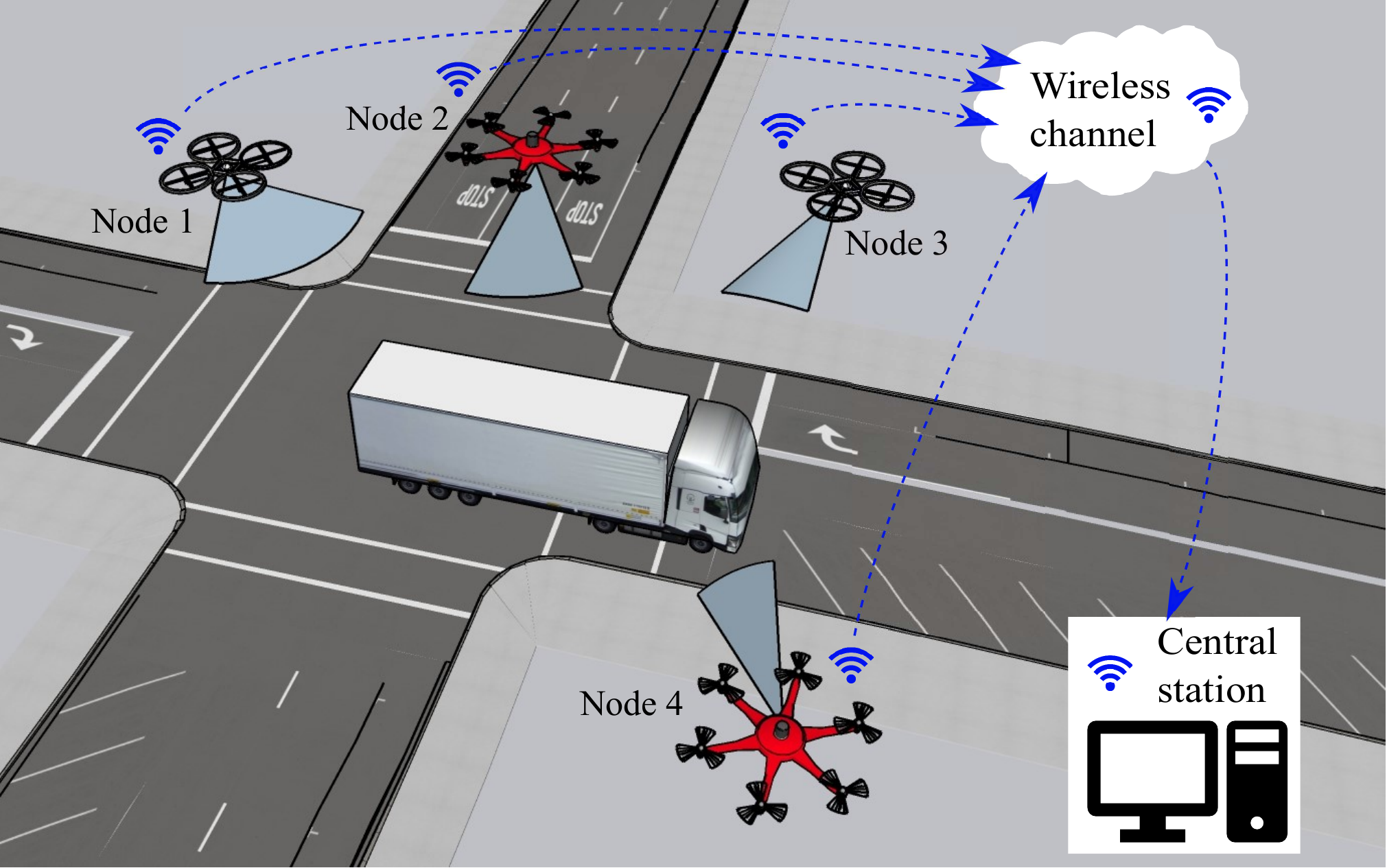}
	\caption{Example of \emph{processing network}: 
	   drones track a moving vehicle in the presence of computation and 
	   communication constraints.
	   Each drone can preprocess the acquired images before transmitting to
	   a central station.\vspace{-4mm}}
	\label{fig:vehicle-tracking}
\end{figure}

The concurrent availability of both efficient wireless connections and low-power processing hardware raises a 
\emph{computation-communication} trade-off. In particular, it is nontrivial whether a node should send raw data to the central station for processing, with negligible computation at the edge but higher communication delays,
or preprocess the data locally, obtaining more accurate and compact information to be transmitted to and processed by the central station.
\autoref{fig:vehicle-tracking} provides an example of this trade-off: 
it depicts a network of drones observing the state of a
truck and transmitting data to a fusion station (bottom-right laptop), in charge of tracking the vehicle.
The sensors may have heterogeneous resources: for instance, the hexarotors (nodes 2 and 4) might have powerful onboard GPU-CPU systems, while the quadrotors (1 and 3) limited processing hardware.
Therefore, some sensors might prefer sending raw data and incur larger communication delays, while other might preprocess the data at the edge. These choices will impact the quality of the truck state estimate: larger computational and communication delays will introduce more uncertainty, hindering the tracking task.
~\autoref{fig:sensorDelaysSubfigures} shows the central station receiving delayed information.

\begin{subfigures}\label{fig:sensorDelaysSubfigures}
	\centering
	\begin{figure}
		\centering
		\includegraphics[width=\columnwidth]{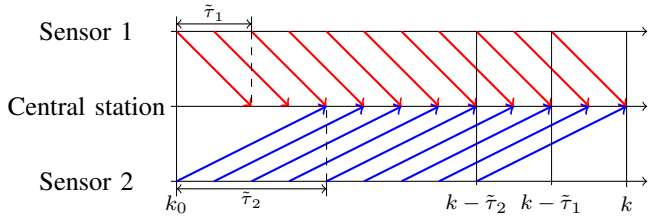}
		\caption{Delayed data transmission with two sensors.}
		\label{fig:sensorDelays}
	\end{figure}
	\begin{figure}
		\vspace{-3mm}
		\centering
		\includegraphics[width=\columnwidth]{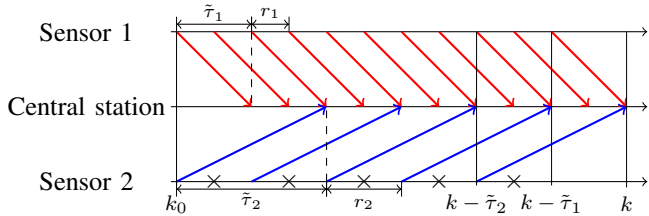}
		\caption{Multi-rate network. Crosses on the bottom axis highlight states observed by sensor 1 but not by sensor 2.}
		\label{fig:sensorDelaysMultirate}
	\end{figure}	
\end{subfigures}

This paper investigates such \tradeoff in a networked system (possibly multi-rate, see acquisition scheme in~\autoref{fig:sensorDelaysMultirate}) monitoring a time-varying phenomenon, modeled as a dynamical system. 

{\bf Related work in IoT} is typically concerned with 
optimizing data transmission through efficient communication policies, mainly using the so-called \emph{Age of Information} (AoI)~\cite{8778671,8469047,10.1145/3209582.3209589}, or the estimation error~\cite{2018arXiv180405618W} as performance metrics. 
In~\cite{8006543}, the \emph{Cost of Update Delays} (CoUD), a non-linear function of AoI, is proposed as an alternative metric. 
The works~\cite{8006543,Bisdikian:2013:QVI:2489253.2489265} introduce the \emph{Value of Information of Update} (VoIU), describing the impact of new samples on the state estimation depending on delay and data correlation.
While these works share some of our motivations, they focus on 
communication/channel properties and on the freshness of the transmitted sensor data, not accounting for 
computation latency nor for the dynamic nature of the monitored system.  
Moreover, our results establish a direct connection between the time constants of the monitored dynamical system and
the optimal computation and communication policy.

{\bf Related work in control and robotics} has mainly focused on other aspects of networked systems.
On the one hand, the control literature extensively investigates the co-design of estimation and control for networked control systems in the presence of communication constraints, such as data rate, quantization, delays, and limited information. The works~\cite{Bertsekas05book,Borkar97cccsp-limitedCommControl,LeNy14tac-limitedCommControl,Shafieepoorfard13cdc-attentionLQG,4118476,6669629}
analyze the trade-offs between communication and control performance (e.g., stability, LQR cost). 
Related work~\cite{Elia01tac-limitedInfoControl,Nair04sicon-rateConstrainedControl} focuses on system stability under limited bandwidth. 
On the other hand, a large body of work studies the design of sensing and actuation. The line of work~\cite{Skelton,Joshi09tsp-sensorSelection,Gupta06automatica,Leny11tac-scheduling,Jawaid15automatica-scheduling,Zhao16cdc-scheduling,Tzoumas16acc-sensorScheduling,Carlone18tro-attentionVIN,Summers16tcns-sensorScheduling,Nozari17acc-scheduling,Summers17arxiv,Summers17arxiv2} investigates 
sensors and actuators selection, in order to maximize some performance metric. For example,~\cite{Skelton} studies a convex LMI formulation, while~\cite{Tatikonda04tac-limitedCommControl,Tzoumas18acc-sLQG} establish a more direct connection between sensing and estimation performance, by proposing co-design approaches for sensing, estimation, and control.
More recent work in robotics~\cite{2019arXiv190205703C} studies computational-offloading in cloud robotics and proposes a learning-based approach.
Departing from the traditional focus on communication constraints, our model also captures \emph{computational delays} and studies the 
fundamental trade-off between computation and communication. 
A few related works investigate computational delays. 
For instance, in~\cite{Tsiatsis2005} the delays of different devices is studied to find an optimal network processing policy.
The work \cite{8757960} characterizes the delays occurring in a network with cloud fog offloading, with case study on Fast Fourier Transform computation,
while~\cite{7800393} investigates multimedia data processing with different architectures.
Contrary to these works, we put an emphasis on the system dynamics,
and inform the computation and communication policy with the parameters of the system being monitored.
For instance, a computational delay of a few seconds is acceptable when monitoring the temperature in a building, 
but it is unacceptable when the goal is to estimate the target of a drone performing aggressive maneuvers 
during a fast pursuit.


\begin{figure}
	\centering
	\begin{minipage}{\linewidth}
		\centering
		\includegraphics[width=.49\textwidth]{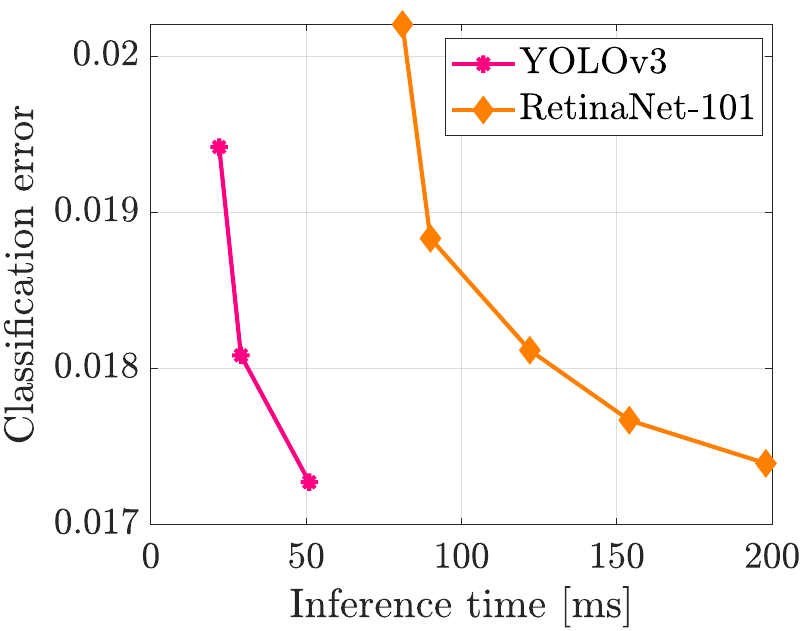}
		\includegraphics[width=.48\textwidth]{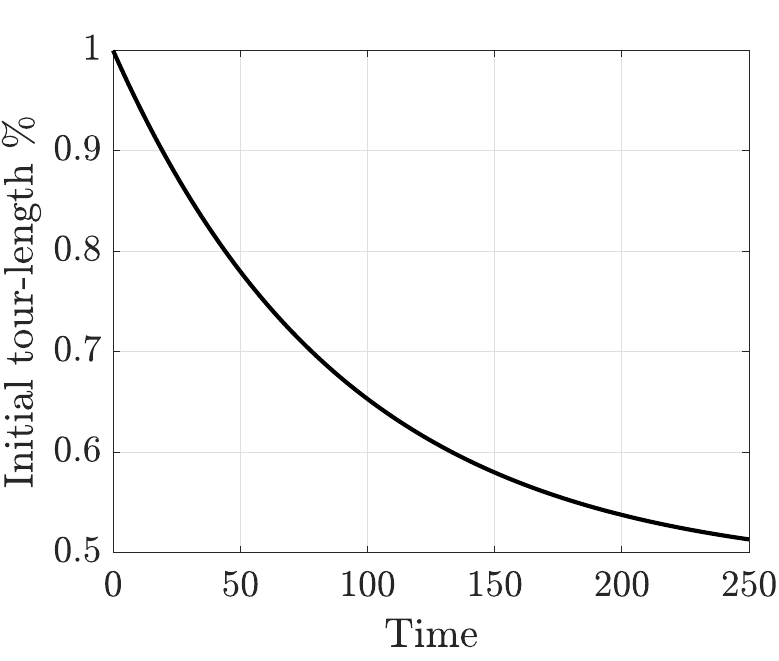}
	\end{minipage}
	\caption{
	    Left: You-Only-Look-Once (YOLO) and RetinaNet are convolutional neural networks that can trade runtime and classification accuracy (errors on y-axis are computed as the inverse of mAP-50 scores).
		Right: Randomized Tour Improvement is a classical greedy algorithm which approximates the optimal tour for the traveling salesman problem, shortening an initial route.
		Adapted from~\cite[Fig.~3]{2018arXiv180402767R},~\cite[Fig.~3]{Zilberstein96ai-anytimeAlgorithms}.}
	\label{fig:anytime-algs}
\end{figure}

{\bf We propose three  contributions}.
First, we introduce a mathematical model for a \emph{processing network}, where nodes 
can perform local computation prior to communicate data to a central station (\autoref{sec:set-up}).
We consider a set of smart sensors in charge of estimating the state of a dynamical system in the presence of communication and computational delays. 
We consider nodes executing \emph{anytime algorithms}~\cite{Zilberstein96ai-anytimeAlgorithms},
\ie whose performance improves with the runtime (as typical descent algorithms in optimization).
Anytime algorithms are popular in computer vision and robotics applications: 
for instance,~\cite{2018arXiv180402767R,NNRuntime} study resource-aware neural networks whose complexity can be traded with runtime (\autoref{fig:anytime-algs}),
~\cite{anytimeKywe} adapts image-processing filters to real-time tasks by varying their kernels,
~\cite{imageCompressionML} proposes a learning-based adaptive image compression,
~\cite{Karaman11ijrr-planning} studies a planning algorithm that asymptotically converges to the optimal solution.
Our model captures such nature of the preprocessing using a computation-dependent measurement noise at each sensor.

Second, we use our model to compute the amount of preprocessing at each node
in the simple case of a homogeneous network (all nodes carry the same computation) 
monitoring a continuous-time scalar linear system (\autoref{sec:hom-net-cont-time}),
and prove that the optimal delay can be computed and characterized analytically.
Also, we show that sending raw data is in general suboptimal in the
presence of communication delays.
Furthermore, we show that in the presence of computational delays --contrarily to conventional wisdom--
using more sensors might hinder estimation performance
(intuitively: processing data from each sensor induces a computational delay, which introduces extra uncertainty in the estimation).

Third, we consider a more realistic heterogeneous network monitoring a 
discrete-time multivariate linear system.
Since using all sensors is not necessarily an optimal policy, we 
extend our formulation to also select an optimal set of sensors, besides deciding the optimal preprocessing at each sensor.
Towards this goal, we first show how to compute an estimation-theoretic cost function to be optimized by the processing network 
(\autoref{sec:discrete-time}), possibly including multi-rate sensors (see~\autoref{fig:sensorDelaysMultirate}).
Then, also leveraging the formal analysis made on homogeneous networks, we propose a greedy algorithm to select the preprocessing and sensors (\autoref{sec:greedy-algorithms}).
Numerical results (\autoref{sec:simulations}) show that (i) our algorithm can indeed compute near-optimal policies, 
(ii) that {using all sensors is in general suboptimal},
and (iii) that the proposed policy can largely improve the network performance. 
We conclude and discuss future work in~\autoref{sec:conclusions}.

\section{Optimal Estimation in Processing Networks: Problem Formulation}
	\label{sec:set-up}

	A \emph{processing network} is a set of interconnected \emph{nodes}\footnote{We often refer to the nodes as \emph{smart sensors} 
	to stress their computational and sensing capabilities.} 
 	that collect
	sensor data and leverage onboard computation to locally preprocess the data before communicating them to a central fusion center.
	In this paper, we consider the case where the network is tasked with
	obtaining an accurate estimate of the state of a time-varying phenomenon observed by the nodes, 
	in the presence of communication and computation latencies.

	\subsection{Anatomy of a Processing Network}\label{sec:set-up-anatomy}

	\begin{description}[leftmargin=0cm]
		\item[Dynamical system:]
		We consider a processing network monitoring a discrete-time dynamic phenomenon described by the following linear time-invariant (LTI) stochastic model:
		\begin{equation}
		x_{k+1} = Ax_k+w_k \label{eq:processModel}
		\end{equation} 
		where $x_k \in \Real{n}$ is the to-be-estimated state of the system at time $k$, 
		$A \in \Real{n\times n}$ is the state matrix, and \revision{$w_k\sim \gauss(0,Q)$} is i.i.d. zero-mean Gaussian white noise 
		with covariance $Q$. 

		\item[Nodes (smart sensors):] 
		The processing network includes \revision{the set of nodes $\sensSet$}. 
		After acquiring raw data, each node may refine them via some local preprocessing. 
		For instance, in the control application of~\autoref{fig:vehicle-tracking}, each drone is a smart sensor that may preprocess raw images to get
		local measurements of the state (the tracked vehicle location and velocity).
		Depending on the available time and computational resources, a sensor may either run different procedures, 
		or adopt an anytime algorithm (e.g., varying the number of visual features~\cite{Hartley04}), to obtain 
		a refined measurement.
		The local measurements produced by all nodes in the network (possibly after some preprocessing) are modeled as:
		\begin{equation}
		\hspace{-1mm}z_k(\delaySetDisc{p}{}) = Cx_k+v_k(\delaySetDisc{p}{}), \quad
		z_k(\delaySetDisc{p}{}) = 
		\begin{bmatrix}
		z_k^{(1)}(\delayDisc{p,1}) \\
		\vdots \\
		z_k^{(\lvert \sensSet \rvert)}(\delayDisc{p,\lvert \sensSet \rvert})
		\end{bmatrix} \label{eq:measurementModel}
		\end{equation}
		where $z_k^{(i)}\in\Real{m_i}$ is the measurement collected at time $k$ by the $i$-th node
		(starting from an initial time $k_0 \le k $),
		$\delayDisc{p,i}$ is the \textit{preprocessing delay} associated with that node,
		\revision{$ C $ describes the state-to-output sensor transformation,}
		and $v_k\sim \gauss(0,R)$ is i.i.d. zero-mean Gaussian noise; $\delaySetDisc{p}{} \doteq \{\delayDisc{p,i}\}_{i\in\sensSet}$ collects all the preprocessing delays, and $z_k$ contains the measurements from all nodes. 
		To capture the anytime nature of the node preprocessing, we model the intensity $R(\delaySetDisc{p}{})$ 
		of the noise $v_{k}$ as a decreasing function of the delays $\delayDisc{p,i}$, that is, the more time a node spends on local preprocessing, the more accurate the resulting data are.
		\revision{In general, nodes with faster processors induce a faster decrease of the uncertainty $ R $, since they 
		can quickly process a larger amount of sensor data (this is formalized in~\autoref{sec:hom-net-cont-time}).}
		
		\item[Communication network:] The nodes transmit preprocessed data to the \emph{central station} for fusion. 

		\revision{
			To account for channel unreliability, we associate to
			the measurement transmitted from the $ i $-th node at time $ k $
			the binary random variable $ \gamma_k^{(i)} \sim \mathcal{B}(\lambda_i) $, which denotes successful reception at the central station.
			Conversely, $ 1-\lambda_i $ is the \emph{packet-loss probability} associated with each transmitted measurement from the $ i $-th node.
			We assume that $ \gamma_k^{(i)} $ and $ \gamma_\ell^{(j)} $ are uncorrelated if $ k\neq\ell $ or $ i\neq j $.
			Finite capacity is modeled as upper bound on the number of data packets per unit time,
			which induces a maximum number of nodes transmitting simultaneously.
		}

		Given limited bandwidth, also data transmission induces a \textit{communication delay} $\delayCommDisc$ (potentially different for each node $i$).
		We consider two possible models for $\delayCommDisc$ as a function of the preprocessing delay $\delayDisc{p,i}$:
		\begin{itemize}
			\item \emph{constant $\delayCommDisc$:} the transmitted number of packets is fixed and does not depend on the amount of preprocessing,
			\revision{but may increase with the dimension of the transmitted data;} 
			\item \emph{decreasing $\delayCommDisc$:} \revision{if nodes \emph{compress} the measurements, a longer preprocessing yields fewer packets to transmit.
			In this case, nodes with more computational resources induce a higher compression rate, leading to a faster decrease of $ \delayCommDisc $.}
		\end{itemize}
		
		\revision{The total delay to preprocess and send the data from the $ i $-th node to the central station is} $ \delayDiscTot{i} \doteq \delayDisc{p,i}+\delayCommDisc$ (see~\autoref{fig:sensorDelays}).
		
		\begin{figure*}
			\centering
			\includegraphics{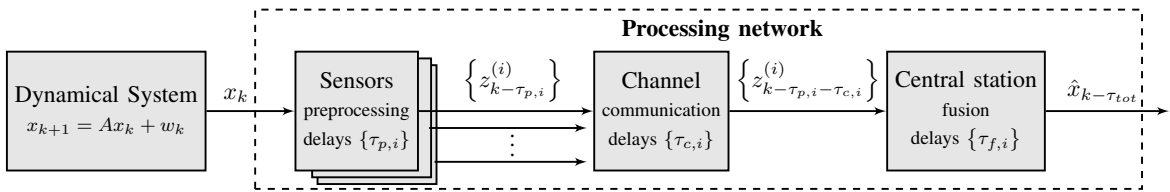}	
			\caption{Block diagram of the processing network with preprocessing, communication and fusion delays.}
			\label{fig:blocks-delays}
		\end{figure*}
		
		\item[Fusion center:] The central station is in charge of fusing all \revision{received} sensor data to compute a state estimate.
		Such centralized processing adds further latency, namely the \textit{fusion delay} $\delayFusDiscTot$, which is the sum of all delays $\delayFusDisc$ required to process the data stream from each node $i$.

		\revision{As for the communication delay, we assume that the fusion delay $\delayFusDisc$ is either constant, or it decreases with the delay $ \delayDisc{p,i} $
		(intuitively, the more preprocessing is done at the node, the less effort is needed for fusion). 
		In the former case, the delay is simply a function of the computational resources at the fusion station.
		In the latter, it might depend on the amount of data compression performed by node $i$.}
		
		\revision{\autoref{fig:blocks-delays} provides an overview of the processing network with the different latency contributions - by node preprocessing, communication, and centralized fusion.
		 As highlighted by the figure, the raw data goes through a number of operations, each one inducing delays.
		 Therefore, the state estimate at time $k$ will not include measurements at time $k$, but only partially outdated measurements collected at time $<k$, due to the delays. 
		 This is formalized in the following definition. } 
		\begin{definition}\label{def:datasets}
			The \emph{processed} dataset at time $k$ is
				\begin{equation}
				\revision{
					\mathcal{Z}_k\!\left(\delaySetDisc{p}{}\right) \!\doteq \!\left\lbrace z_{\ell_i}^{(i)}(\delayDisc{p,i}) \!: \ell_i\!\in\! [k_0,k\!-\!\delayDiscTot{i}\!-\!\delayFusDiscTot], {\gamma_{\ell_i}^{(i)} \!= \!1}\right\rbrace_{i\in\sensSet}
				}
				\end{equation}
				
		\end{definition}
		In words, the processed dataset includes all \revision{correctly received} measurements \emph{except} the most recent ones collected during 
		communication and data processing, \ie during the latest $\delayDisc{p,i}+\delayCommDisc+\delayFusDiscTot$ timestamps.

	\end{description}

		\subsection{Optimal Estimation in Processing Networks}

		In this section we first motivate our interest in optimizing for the amount of preprocessing at each node and 
		the need to select a subset of nodes. Then, we provide a suitable performance metric to measure estimation performance.
		Finally, we put together these elements to formulate the problem of optimal estimation in processing networks 
		(\cref{prob:time-invariant-P-opt}).

		\textbf{Preprocessing selection:} While the sensor data might be received and fused with some (computation and communication) delay, 
		we are interested in obtaining an accurate state estimate at the current time $k$; this entails fusing 
		sensor information $\mathcal{Z}_k(\delaySetDisc{p}{})$ (partially outdated, due to the computation and communication delays) with the open-loop system prediction in~\eqref{eq:processModel}. 
		These delays create a nontrivial \tradeoff: 
		is it best to transmit raw sensor data and incur larger communication and fusion delays,
		or to perform more preprocessing at the edge and transmit more refined (less noisy and more compressed) estimates? 
		For instance, consider again~\autoref{fig:vehicle-tracking} where robots compute local estimates from images. 
		Consider the case in which the local estimates are computed using local features extracted from the images, as 
		common in geometric computer vision~\cite{Hartley00}.
		Each extracted feature both enhances node-side accuracy and possibly reduces transmission and fusion latency. However, feature extraction entails some preprocessing latency at the edge.
		A trade-off emerges: on one hand, many features cause a delayed prediction; on the other hand, few provide poor accuracy.
		An \emph{optimal estimation} policy has to decide the preprocessing at each node in a way to maximize the 
		final estimation accuracy.
		
		\textbf{Sensor selection:}  In addition, the fusion latency increases with the number of sensors transmitting data.
		As a consequence, the amount of open-loop prediction required to compensate for the fusion delay increases with the 
		number of nodes, hence adding more sensors does not necessarily improve performance. 
		Therefore, in order to maximize the estimation accuracy, the network can also 
		decide to use only a subset of sensors $\mathcal{S}\subseteq\sensSet$ (below we refer to those as \emph{active nodes}), 
		such that the state estimate is computed using only data from those sensors, $\mathcal{Z}_k\left(\mathcal{S},\delaySetDisc{p}{\mathcal{S}}\right)\subseteq \mathcal{Z}_k\left(\delaySetDisc{p}{}\right)$, where $\delaySetDisc{p}{\mathcal{S}}$ denotes the set of computational delays for the active nodes.

		\textbf{Performance metric:} 
		In a state estimation problem,
		the estimation performance can be measured as the Mean Squared Error (MSE) of an optimal estimator, \ie 
		$\mbox{Var}\left(x_{k}-\hat{x}_{k}(\delaySetDisc{p}{\mathcal{S}})\right) $, where $\hat{x}_{k}\left(\delaySetDisc{p}{\mathcal{S}}\right) \doteq g\left(\mathcal{Z}_k\left(\mathcal{S},\delaySetDisc{p}{\mathcal{S}}\right)\right)$ is the state estimate from an optimal estimator that uses the reduced processed dataset $\mathcal{Z}_k\left(\mathcal{S},\delaySetDisc{p}{\mathcal{S}}\right)$.
		\revision{We use the Kalman filter, which is the optimal MSE estimator for linear systems with Gaussian noise.}
		However, the optimal filter comes with the nuisance of time variance
		\revision{
			and dependence on the specific packet arrivals, and convergence analysis is not feasible (cf.~\cite{schenatoKalmanFusion,sinopoliKalman}).
			Instead, we resort to the (suboptimal) filter with constant gains (\ie not depending on the arrival-sequence instance),
			and address the steady-state expected performance.
		}
	
		\textbf{Problem formulation:}  We are now ready to formalize the problem of \problemName.
	
		\begin{prob}[\problemName]\label{prob:time-invariant-P-opt}
			Given system~\eqref{eq:processModel} with available sensor set $\sensSet$ and measurement model~\eqref{eq:measurementModel}, find 
			the optimal sensor subset $\mathcal{S}$ (the \emph{active} sensors) and preprocessing delays $\delaySetDisc{p}{\mathcal{S}}$ that minimize the steady-state expected estimation error variance:
			\begin{equation}
			\argmin_{\small\begin{array}{c}
				\mathcal{S}\subseteq\sensSet\\
				\delaySetDisc{p}{\mathcal{S}} = \{\delayDisc{p,i}\}_{i\in\mathcal{S}}\in\mathbb{N}^{|\mathcal{S}|}
				\end{array}} \; \tr{\Psteady\left(\delaySetDisc{p}{\mathcal{S}}\right)}
			\label{prob-2} 
			\end{equation}
			where the total delay $\delayDisc{tot}$ is defined as
			\begin{equation}
				\delayDisc{\textit{tot}}\doteq \underbrace{\min_{i\in\mathcal{S}}\delayDiscTot{i}}_{\doteq\,\delayDiscTot{\textit{min}}} + \underbrace{\sum_{i\in\mathcal{S}}\delayFusDisc}_{\doteq \delayFusDiscTot}
				\label{total-delay}
			\end{equation}
			\revision{
				and the steady-state expected error covariance is
				\begin{equation}\label{eq:errorCovariance}
					\Psteady\left(\delaySetDisc{p}{\mathcal{S}}\right) \doteq \displaystyle\lim_{k\rightarrow +\infty}\!\!
					\mathbb{E}\left[\mbox{Var}\left(x_{k}-\hat{x}_{k}\left(\delaySetDisc{p}{\mathcal{S}}\right)\right)\right]
				\end{equation}
				where the expectation is taken with respect to the sequence  $ \{\gamma_k^{(i)} \, \forall k\ge k_0, \forall i\in\mathcal{S}\} $.\footnote{\revision{We assume that the packet-loss probabilities are small enough so as that the steady-state estimator with constant gains exists.}}
			}
			The delay $\delayDisc{\textit{tot}}$ accounts for the fact that, because of delays, the steady-state estimate relies on partially outdated measurements: 
			$ \delayDiscTot{\textit{min}} $ is the time it takes to receive all processed data from the sensors (including the freshest data collected in $ \mathcal{Z}_k(\mathcal{S},\delaySetDisc{p}{\mathcal{S}}) $), while $\delayFusDiscTot$ is the time it takes to 
			fuse them at the central station.
		\end{prob}
	
		\begin{rem}[Parallel data collection \versus sequential fusion]\label{rem:totalDelay}
			The delay $\delayDiscTot{min}$ is computed as the minimum over the active sensors, as these work in ``parallel'', while the fusion delay $\delayFusDiscTot$ is additive, as in general the fusion center processes all data sequentially. Therefore, the latter is more sensitive to variations of computational delays. Besides, the fusion delay increases with the number of sensors, possibly limiting the network scalability.
		\end{rem}
	
		\begin{rem}[Comparison with sensor selection]\label{rem:constantDelaysInProblem}
			\revision{
				The problem formulation~\eqref{prob-2} differs from standard sensor selection, where each sensor comes with a cost and one aims at maximizing performance under cost constraints~\cite{2018arXiv180208376T,Skelton,Joshi09tsp-sensorSelection,Tzoumas16acc-sensorScheduling}.
				In particular, we are interested in solving the \tradeoff, binding the sensor selection to that of suitable preprocessing delays.
				Therefore, in our setup, rather than associating a cost to each sensor, the 
				penalty in using a sensor is captured by the amount of computation and delay it induces at the fusion station.
			}
		\end{rem}
	
		From now on, we write $\delayDisc{i} = \delayDisc{p,i}$ and $ \delaySetDisc{}{}=\delaySetDisc{p}{\mathcal{S}} $ for the sake of readability.
		Before designing algorithms to solve~\cref{prob:time-invariant-P-opt} (Sections~\ref{sec:discrete-time}--\ref{sec:greedy-algorithms}), 
		we analyze its continuous-time counterpart, which can be solved analytically when the set of sensors is fixed and homogeneous.
		Such simplified approach provides useful insights on the cost function in~\eqref{prob-2}, which are used to tackle the discrete-time case in~\autoref{sec:greedy-algorithms}.

\section{Continuous-time Scalar Analysis}\label{sec:hom-net-cont-time}

We now consider a continuous-time scalar system monitored by a homogeneous network, which is composed of {$ \sensNum $ independent} sensors with equal preprocessing, communication, and fusion delays. In this section we only solve~\cref{prob:time-invariant-P-opt} with respect to the preprocessing delay $ \delayCont $, while the need for sensor selection is motivated in~\autoref{sec:hom-net-cont-time-psteady(N)} with a numerical example.
\revision{
	Also, we assume infinite channel capacity and reliable communication for the sake of simplicity, relaxing such assumptions in Sections~\ref{sec:discrete-time}--\ref{sec:greedy-algorithms}.}
Consider the following continuous-time scalar system:
\begin{equation}\label{eq:processModelCont}
dx_t = ax_tdt + dw_t \qquad  dw_t\sim\gauss(0,\sigma^2_wdt) 
\end{equation}
and the homogeneous-network model
\begin{equation}\label{homogeneous-network}
z_t(\delayCont) = \mathbbm{1}_\sensNum{c}\,x_t + v_t(\delayCont) \quad v_t(\delayCont) \sim \gauss\left(0,I_\sensNum\sigma^2_v(\delayCont)\right)
\end{equation}
where \revision{$ a $ describes the state dynamics, $w_t$ is the process noise, and $\sigma^2_w$ is its variance;
$ \mathbbm{1}_\sensNum $ is the vector of ones and $ c $ and $ \sigma^2_v(\delayCont) $ are scalars modeling the noisy state-output transformation of each sensor. The symbol $I_\sensNum$ denotes the identity matrix of size $\sensNum$.
The vector $ z_t(\delayCont)\in\Real{\sensNum} $ collects all the measurements from the $ \sensNum $ sensors and $ v_t(\delayCont) $ is the overall measurement noise, with covariance matrix $ I_\sensNum \sigma^2_v(\delayCont) $.}

\revision{The anytime nature of the local processing at each node is captured by making the measurement noise covariance $\sigma^2_v(\delayCont)$  
a decreasing function  of the preprocessing delay $\delayCont$.}
Since the Least-Squares-estimation-error variance is inversely proportional to the number of collected samples at each node, 
we opted for the following model:
\begin{equation}\label{R-model-cont-time-hom-net}
\sigma^2_v(\delayCont) = \frac{b}{\delayCont} \qquad b > 0
\end{equation}
\revision{The coefficient $ b $ depends on the node parameters: on the one hand, nodes with large computational resources improve quickly their output accuracy, yielding a small $ b $; on the other hand, if the collected raw data are heavy (e.g., images), refining them takes more time, inducing a larger $b$.}
Communication and fusion delays $ \delayCommCont, \delayFusCont $ are given as: \\
\begin{subequations}\label{comm-fus-del}
	\begin{tabularx}{.95\linewidth}{p{.4\linewidth}p{.48\linewidth}}
		\begin{equation}\label{comm-fus-del-const}
			\hspace{-2.5cm}
			\mbox{\hspace{-1cm}constant}:\begin{cases}
			\delayCommCont\equiv\delayCont_{c} \\
			\delayFusCont\equiv\delayCont_{f}
			\end{cases}
			\hspace{-2.4cm}
		\end{equation}
		&
		\begin{equation}\label{comm-fus-del-var}
			\hspace{-20mm}
			\delayCont\mbox{-varying}:\begin{cases}
			\delayCommCont=\frac{c}{\delayCont} \\
			\delayFusCont=\frac{f}{\delayCont}
			\end{cases}
			\hspace{-1.3cm}
		\end{equation}
	\end{tabularx}
\end{subequations}
\revision{where the delays are either given constants $\delayCont_{c},\delayCont_{f}$ as in eq.~\eqref{comm-fus-del-const}, 
or are inversely proportional to the preprocessing delay (with given coefficients $c$ and $f$), as in eq.~\eqref{comm-fus-del-var}.}
\revision{
	We assume $ \delayCont_{c},\ \delayCont_{f}, \ c $ and $ f $ to be positive and known.
	Both communication and fusion compression coefficients $ c $ and $ f $ increase with the dimension of the raw measurements.
	Conversely, sensors with more computational resources can compress faster and induce smaller coefficients.
}
\begin{rem}
	While the models in~\eqref{comm-fus-del-var} are mainly used for mathematical convenience, in a real setup the
	compression functions might be learned or estimated from data, e.g.~\cite{imageCompressionML}.
\end{rem}

In a homogeneous network, the total delay simplifies to (cf.~\eqref{total-delay} when all nodes are active and have the same delays)
\begin{equation}
\delayCont_{\textit{tot}} = \delayCont+\delayCommCont+\delayFusCont \sensNum\label{total-delay-hom}
\end{equation}
Note the linear dependence of the total fusion delay on the sensor amount $ \sensNum $.
In such setup,~\cref{prob:time-invariant-P-opt} simplifies to the following formulation, where we neglect sensor selection to focus on the computation of the optimal preprocessing delay.

\begin{prob}[Optimal Estimation in Continuous-time Processing Network]\label{prob:hom-network}
	Given system~\eqref{eq:processModelCont} with $ \sensNum $ identical sensors and measurement model~\eqref{homogeneous-network}, find the optimal preprocessing delay $\delayCont$ that minimizes the steady-state expected estimation error variance:
	\begin{equation}
	\argmin_{\delayCont \in\Realp{}}\;\psteady(\delayCont)
	\end{equation}
\end{prob}

It turns out that~\cref{prob:hom-network} has a unique analytical solution, as formalized by the following theorem.

\begin{thm}[Optimal preprocessing for continuous-time homogeneous network,~\cite{ballotta19ifac}]\label{thm-scalar-total}
	Consider the LTI system~\eqref{eq:processModelCont}--\eqref{homogeneous-network} with
	measurement noise variance \revision{$\sigma^2_v(\delayCont)$} as per~\eqref{R-model-cont-time-hom-net},
	communication and fusion delays $\delayCommCont$, $\delayFusCont$ as per~\eqref{comm-fus-del-const} or~\eqref{comm-fus-del-var}
	and initial condition ${x_{t_0}\sim\mathcal{N}(\mu_0,p_0)} $. Assume $ \hat{x}_{t}(\delayCont) $ is the Kalman-filter estimate at time $ t $ given measurements collected until time $ t-\delayCont_{\textit{tot}} $. Then, the steady-state error variance $ \psteady(\delayCont) $ is
	\begin{equation}\label{eq:psteadyCont}
	\psteady(\delayCont) = \underbrace{\mbox{e}^{2a \delayCont_{\textit{tot}} }p_{\infty}(\delayCont)}_{\doteq f(\delayCont)}+\underbrace{\frac{\sigma^2_w}{2a}\left(\mbox{e}^{2a \delayCont_{\textit{tot}} }-1\right)}_{\doteq q(\delayCont)}
	\end{equation}
	where
	\begin{equation}
	p_{\infty}(\delayCont) = \frac{\tilde{b}}{\delayCont}\left(a+\sqrt{a^2+\frac{\sigma^2_w}{\tilde{b}}\delayCont}\right) \quad
	\revision{\tilde{b}\doteq\frac{b}{\sensNum c^2}}
	\end{equation}
	with limits
	\begin{align}
	\!\begin{split}
	\!\lim_{\delayCont \rightarrow 0^+} \psteady(\delayCont) \!=\! \lim_{\delayCont \rightarrow +\infty} \psteady(\delayCont) \!=\! \begin{cases}
	+\infty, &\! a\ge 0\\
	\dfrac{\sigma^2_w}{2|a|}, &\! a<0 \end{cases}
	\end{split}
	\label{error-variance-limits-scalar}
	\end{align} 
	and has a unique global minimum at $\delayCont_{{\textit{opt}}}>0$.
	\revision{Finally, when the delays $\delayCommCont$ and $ \delayFusCont $ are constant, as per~\eqref{comm-fus-del-const}, $\delayCont_{\textit{{\textit{opt}}}}$ satisfies:}
	\begin{equation}
	\frac{\sigma^2_w}{\tilde{b}}\delayCont_{\textit{{\textit{opt}}}}^3 = -a^2\delayCont_{\textit{{\textit{opt}}}}^2+\frac{1}{4} \label{3rd-degree-eq-scalar-general}
	\end{equation}
\end{thm}

\begin{proof}
	See appendix~\ref{app:proof-thm-scalar-total} in the supplementary material.
\end{proof}

\begin{figure}
	\centering
	\includegraphics[width=0.7\linewidth]{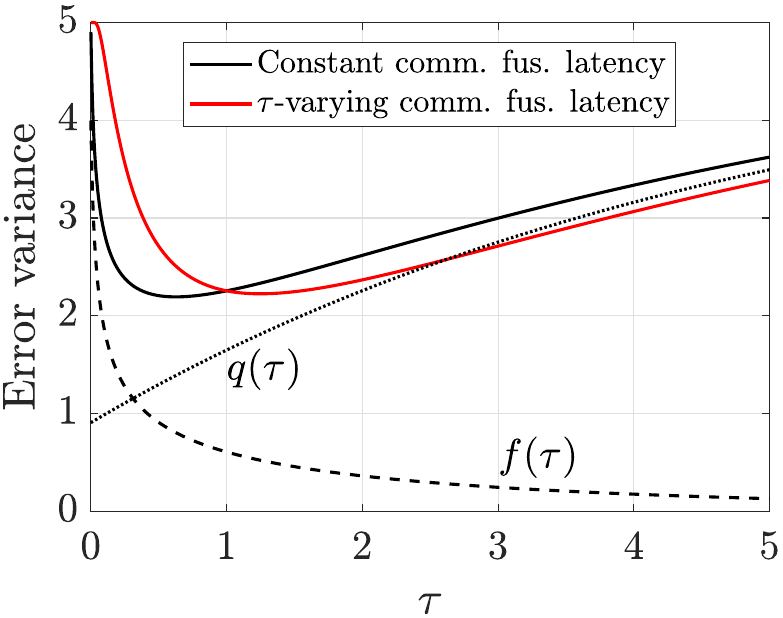}
	\caption{Representation of variance $ \psteady(\delayCont) $.} 
	\label{fig:poftautwoparts}
\end{figure}

The proof exploits quasi-convexity of the expected variance $ \psteady(\delayCont) $.
\autoref{fig:poftautwoparts} illustrates the cost function with the two models for communication and fusion delays (black for constant and red for $\delayCont$-varying) for an asymptotically stable system; for the former, the contributions due to estimation $ f(\delayCont) $  and \revision{ to process noise $ q(\delayCont) $ as given in~\eqref{eq:psteadyCont} are shown as dashed and dotted lines, respectively.} 
The solid curves cross, the red one being lower for $ \delayCont > 1 $, suggesting that compressing data at fixed rate is convenient if the preprocessing delay is kept below a certain threshold. 

Eq.~\eqref{3rd-degree-eq-scalar-general} allows for a closed-form computation of $\delayCont_{\textit{{\textit{opt}}}}$ if model~\eqref{comm-fus-del-const} holds. In general, being the variance $ \psteady(\delayCont) $ quasi-convex, a numerical solution can be computed efficiently.
Optimal preprocessing with other models for $ \sigma^2_v(\delayCont) $ is discussed in Appendix~\ref{sec:other-functions} in the supplementary material.

\begin{ex}[Brownian systems]
	One interesting case arises when the system~\eqref{eq:processModelCont} describes a Brownian motion:
	\begin{equation}
	dx_t = dw_t \label{scalar-system-a-eq-0}
	\end{equation}
	In this situation, the optimal delay has a simple expression.
	
	\begin{cor}[Brownian motion]
			Given system~\eqref{scalar-system-a-eq-0} and \eqref{homogeneous-network} and hypotheses as per~\cref{thm-scalar-total}, the steady-state expected error variance has the following expression:
			\begin{equation}
			\psteady(\delayCont) = \underbrace{\sqrt{\dfrac{\tilde{b}\sigma^2_w}{\delayCont}}}_{f(\delayCont)} + \underbrace{\sigma^2_w\delayCont}_{q(\delayCont)}
			\end{equation}
			admitting the unique global minimum
			\begin{equation}
			\delayCont^\text{B}_{\textit{opt}} = \sqrt[3]{\dfrac{\tilde{b}}{4\sigma^2_w}}
			\label{tmin-a-eq-0}
			\end{equation}
	\end{cor}
	
	The cubic root in \eqref{tmin-a-eq-0} strongly reduces the parametric sensitivity of $ \delayCont^\text{B}_{\textit{opt}}$, which may help with uncertain models.
\end{ex}

\subsection{\revision{Sensitivity of Optimal Preprocessing}} \label{sec:param-dependence}

Based on eq.~\eqref{3rd-degree-eq-scalar-general}, with constant delays~\eqref{comm-fus-del-const} the behavior of the optimal delay $\delayCont_{\textit{opt}}$ can be analyzed as a function of the system parameters. 
In particular, $ \sigma^2_w $ and $\tilde{b}$ do not act independently, so we can focus on their ratio $ \rho\doteq\nicefrac{\sigma^2_w}{\tilde{b}} $.

\begin{prop}\label{prop-params}
	Let $ \delayCont_{\textit{opt}} $ be the solution of~\eqref{3rd-degree-eq-scalar-general} with $ \delayCommCont $, $ \delayFusCont $ as per~\eqref{comm-fus-del-const};
	then, $ \delayCont_{\textit{opt}} $ is strictly decreasing with $\rho$ and $a^2$. 
\end{prop}

\begin{proof}
	See Appendix~\ref{app:proof-prop-params} in the supplementary material.
\end{proof}	
	
On the one hand,~\cref{prop-params} states that it is more convenient to reduce the preprocessing for 
``unpredictable systems'', characterized by fast dynamics or large process noise.
On the other hand, if the sensor noise is large, it is better to further refine the measurements,
which explains why $\delayCont_{\textit{opt}}$ grows with $b$.
\revision{Also, since the parameter $\tilde{b}$ is inversely proportional to the number of sensors $\sensNum$,
then $\delayCont_{\textit{opt}}$ also decreases with $\sensNum$:}
the more data are provided, the less preprocessing is needed to extract accurate information.\\
\autoref{fig:params-dependence} shows the typical behaviour of $\delayCont_{\textit{opt}}$ with respect to the system parameters.

\begin{figure}
	\centering
	\includegraphics[width=0.7\linewidth]{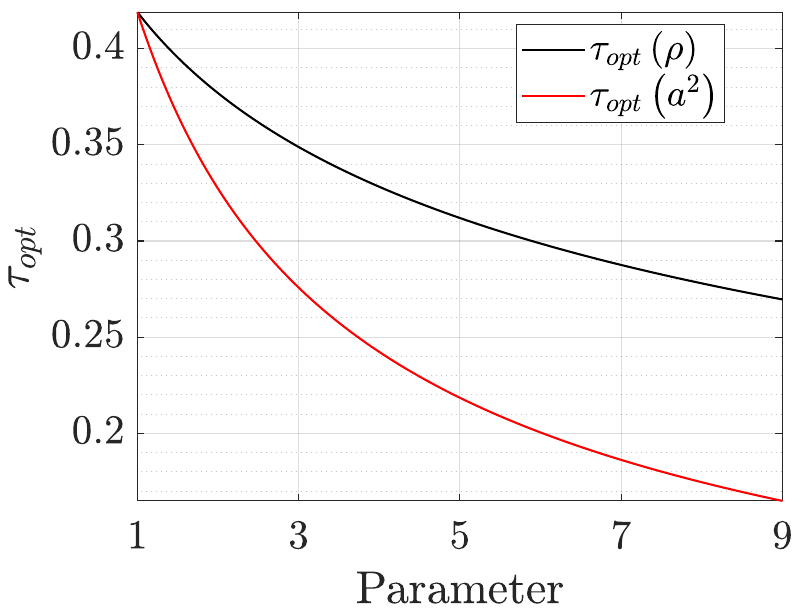}
	\caption{Optimal delay $\delayCont_{\textit{opt}}$ as a function of $ \rho $ ($ a^2=1 $) and $ a^2 $ ($ \rho=1 $).}
	\label{fig:params-dependence}
\end{figure}

\begin{rem}(Insights from continuous-time scalar case)\label{rem:insights}
	The analysis on continuous-time homogeneous networks yields two important insights.
	Firstly, the cost function is quasi-convex. This is exploited in~\autoref{sec:comp-wise-descent} to design a descent strategy
	optimizing the preprocessing delays of a given sensor subset.
	\revision{Secondly, using all sensors is not necessarily an optimal strategy, and -- in the presence of fusion delays -- 
		using a subset of the sensors leads to optimal estimation performance.}
	\revision{This justifies our formulation in~\cref{prob:time-invariant-P-opt}
		(with the selection of a suitable sensor subset) and motivates the design of the 
		greedy sensor selection approach in~\autoref{sec:sensor-selection}.}
\end{rem}

\subsection{\revision{Homogeneous Network: Performance \versus Sensors}}\label{sec:hom-net-cont-time-psteady(N)}

\begin{figure}
	\centering
	\begin{tikzpicture}[scale=.7]
	\begin{axis}[grid=major,domain=1:10,ymax=3.4,ylabel=$\psteady$,xlabel=Number of sensors $ \cardS $,xtick={1,2,3,4,5,6,7,8,9,10},legend style={at={(.2,.9)},anchor=west}]
	\addplot+[mark=square*,mark size=3,mark options={fill=black},samples=10,color=black,only marks] {1/10*exp(-2*.1-2*.1-2*0.02*x)*(1/.1/x)*(-1+sqrt(1+x*10*.1*10))-10/2*(exp(-2*.1-2*.1-2*0.02*x)-1)};
	\addlegendentry{W/ fusion delay}
	\addplot+[mark=square*,mark size=3,mark options={fill=red},samples=10,color=red,only marks]
	{1/10*exp(-2*.1-2*.1)*(1/.1/x)*(-1+sqrt(1+x*10*.1*10))-10/2*(exp(-2*.1-2*.1)-1)};
	\addlegendentry{W/o fusion delay}
	\end{axis}
	\end{tikzpicture}
	\caption{Variance $\psteady(\cardS)$ with fixed delays and varying number of sensors. The fusion delays are $\delayFusCont\equiv0.02$ (black) and $\delayFusCont\equiv0$ (red).}
	\label{fig:p-of-n-scalar}
\end{figure}

\revision{According to~\eqref{total-delay-hom}, the total delay $ \delayCont_{\textit{tot}} $ depends linearly on the sensor amount $ \sensNum $, being the fusion delay sensor-wise additive (cf.~\cref{rem:totalDelay}).
Therefore, if $ \psteady $ is seen as a function of the number of sensors $ \cardS\in \{1,\dots,\sensNum\} $ (having fixed $\delayCont$, $\delayCommCont$ and $\delayFusCont$), then $ \psteady(\cardS) $ has the same structure of $ \psteady(\delayCont) $ when communication and fusion delays are constant,
and can be minimized analogously (on a discrete domain).}\\
~\autoref{fig:p-of-n-scalar} shows the expected estimation error variance as a function of the sensor amount.
The red marks shows that in the absence of fusion delays the error decreases with the number of sensors.
However, in the realistic case with non-negligible fusion delays (black marks), using more sensors might hinder performance.

\section{Discrete-time Analysis}\label{sec:discrete-time}

This section addresses the general discrete-time, multidimensional formulation in~\cref{prob:time-invariant-P-opt}.
In discrete time, the delays $ \delayDisc{} $ are expressed in time steps with respect to the sampling period $ \Delta $.

~\cref{prob:time-invariant-P-opt} cannot be solved analytically, due to its combinatorial nature.
Also, in general the \revision{cost $ \tr{\Psteady(\delaySetDisc{}{})} $} cannot be computed in closed form, since it derives from the solution of a Riccati equation. 
To make things even more complicated, given a sensor subset, the structure of the cost function also depends \revision{on how} the delays are sorted.

To circumvent these issues, we propose a greedy selection algorithm. 
We do this in two steps. In this section we describe a procedure (based on~\cite{schenato2008}) to compute 
\revision{the cost function in~\eqref{prob-2} (and in particular the steady-state expected covariance $\Psteady(\delaySetDisc{}{})$) 
for a given set of sensors and given preprocessing delays}.
Then, the greedy algorithm that selects sensors and computes the optimal preprocessing is presented in~\autoref{sec:greedy-algorithms}.

\subsection{\revision{Steady-state Covariance Computation}}

\revision{This section shows how to compute the steady-state expected covariance for a given choice of the active sensors  $ \mathcal{S} \subseteq \sensSet $
and given preprocessing delays.
For notational convenience and without loss of generality, we label the active sensors as $\mathcal{S} = \{1, 2, \ldots, \cardS\}$
and denote the corresponding delays as $\delayDisc{1}, \delayDisc{2},\ldots, \delayDisc{\cardS}$. 
Finally, we sort the sensors are discussed below.
}
\begin{ass}[Sensor sorting]\label{ass:sensorSorting}
	The sensors in $ \mathcal{S} $ are labeled according to $ \delayDiscTot{i-1} \le \delayDiscTot{i}, \, i = 2,...,\cardS$ \revision{(cf.~\autoref{fig:sensorDelaysSubfigures})}.
\end{ass}

\revision{\cref{ass:sensorSorting} states that, if $ i < j $, the $ i $-th sensor has its data received at fusion station in shorter time than the $ j $-th sensor, and therefore it provides fresher data.}

\revision{We now provide a procedure which, given sensor delays and parameters, computes the steady-state expected covariance $\Psteady(\delaySetDisc{}{})$ 
(and hence the cost $ \tr{\Psteady(\delaySetDisc{}{})} $ in~\eqref{prob-2}), and is exploited by the algorithm in~\autoref{sec:comp-wise-descent} to assess the performance of sensor subsets.}
The following is \emph{not} a closed-form \revision{--but rather an iterative--} computation.
\revision{We outline the procedure and then provide an illustration of the procedure with $ \cardS = 3 $ active sensors using \autoref{fig:covarianceComputation}.}

We start by expanding the matrices that describe the measurement model~\eqref{eq:measurementModel} as
\begin{equation}
C = \left[ C_1^T
\dots
C_\cardS^T \right]^T \
R(\delaySetDisc{}{}) = \mbox{diag}(R_1(\delayDisc{1}),...,R_\cardS(\delayDisc{\cardS}))
\label{sensor-fusion-model}
\end{equation}
\revision{
	where we assume independent sensors with state-output matrix $ C_i $ and covariance $ R_i(\delayDisc{i}) = \nicefrac{b_i}{\delayDisc{i}}I_{m_i} $ (cf.~\eqref{R-model-cont-time-hom-net}).
}

\revision{Before introducing the key result in~\cref{thm:costFunction} below, we
introduce  a number of definitions associated with the Kalman filter with constant gains (cf.~\cite{schenatoKalmanFusion,sinopoliKalman}), which will be 
necessary for the statement of the theorem.}

\begin{definition}\label{def:defKF}
\revision{We define the following operations associated with the Kalman filter with constant gains and acting on the 
extended state estimate covariance matrix $P$:}
\begin{itemize}[leftmargin=*]
	
	\item \textit{Multi-step prediction} with $ \delayDisc{} > 0$ steps: 
	\begin{equation}\label{eq:KalmanOpenLoop}
	\timeUp^\delayDisc{}(P) \doteq \underbrace{\timeUp\circ...\circ\timeUp}_{\delayDisc{} \small\mbox{ times}}(P), \quad \timeUp(P) \doteq APA^T + Q
	\end{equation}
	
	\item \textit{Measurement update} with data acquired at \revision{time $k-\delayFusDiscTot -\delay $}:
	\begin{equation}\label{eq:measUpdate}
	\update\left( P,\delaySetDisc{\delay}{}\right) \doteq \left( P^{-1} + \tilde{\Gamma}(\delaySetDisc{\delay}{})\right) ^{-1}
	\end{equation}
	\revision{
	for any delay $ \delta $, where the information matrix of the processed data when the Kalman gains are constant is:\footnote{All updates use the Kalman filter in information form to handle the fusion more easily. This is also useful if sensor measurements have infinite variance at some locations. Having independent sensors yields a nice expression for $ \tilde{\Gamma} $, where each contribution is visible and disjoint from the others.} 
		\begin{equation}\label{eq:informationMatrixPktLoss}
			\tilde{\Gamma}(\delaySetDisc{\delay}{}) =
			\sum_{i \in \mathcal{S}(\delay)} \lambda_i\left[\Gamma_i - \Gamma_i\left(\dfrac{P^{-1}}{1-\lambda_i}+\Gamma_i\right)^{-1}\Gamma_i\right]
		\end{equation}
		In the previous expression, $ \Gamma_i = C_i^T(R_i(\delayDisc{i}))^{-1}C_i $ is the information matrix of the $ i $-th sensor 
		and $1-\lambda_i$ is its packet-loss probability
		(for details about the derivation of $ \tilde{\Gamma} $ with packet loss, see Appendix~\ref{app:pktLoss} in the supplementary material).
		We remark that the update is restricted to the sensors from which we have received the measurements by time $k - \delayFusDiscTot$.
		Formally, this set of processed sensors is:
		\begin{equation}\label{eq:delayedSensorTimeVar}
		\mathcal{S}(\delay)\doteq\left\lbrace i \in\mathcal{S} : \delayDiscTot{i} \le \delay \right\rbrace
		\end{equation}
		and their preprocessing delays are in $ \delaySetDisc{\delay}{} \doteq \{\delayDisc{i}\}_{i\in\mathcal{S}(\delay)}$.
	}
	
	\item \textit{One-step KF iteration} with data acquired at \revision{time $k-\delayFusDiscTot - \delay $}:
	\begin{equation}\label{eq:KalmanUpdate}
	\iteration\left( P,\delaySetDisc{\delay}{}\right) \doteq \timeUp\circ\update\left( P,\delaySetDisc{\delay}{}\right)
	\end{equation}
	
	\item \textit{Multi-step KF iteration} with data acquired in the time interval $ [k-\delayFusDiscTot-\delay_i + 1, \, k-\delayFusDiscTot -\delay_j] $ for any delays $\delay_i > \delay_j$:
	\begin{equation}\label{eq:KalmanMultiUpdate}
	\iteration^{\delay_i-\delay_j}\left( P,\delaySetDisc{\delay_i-1}{}\right) \doteq \iteration\left( ...\, \iteration\left( P,\delaySetDisc{\delay_i-1}{}\right), ..., \delaySetDisc{\delay_j}{}\right)
	\end{equation}
	\revision{
		where the one-step KF iterations may involve different subsets of active sensors, according to their delays.
	}

\end{itemize}
\end{definition}

Then we have the following result.

\begin{figure}
	\centering
	\includegraphics{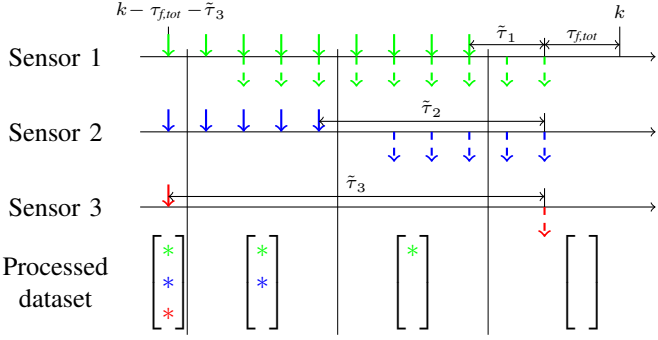}
	\caption{Estimation at time $ k $. Solid and dashed arrows show acquired and received (by central station) data, respectively. Colored stars represent sensor data which are available in the processed dataset.}
	\label{fig:covarianceComputation}
\end{figure}
	
\begin{thm}\label{thm:costFunction}
	\revision{Using the terminology and notation in~\cref{def:defKF},}
	the cost function in~\eqref{prob-2} is given by the trace of
	\begin{equation}
	\Psteady(\delaySetDisc{}{})= \timeUp^{\delayDisc{\textit{pred}}}\left(\iteration^{\delayDiscTot{\cardS}-\delayDiscTot{1}}\left(P_{\infty}\left(\delaySetDisc{}{}\right),\delaySetDisc{\delayDiscTot{\cardS}-1}{}\right)\right)
	\label{cost-function-open-loop}
	\end{equation}
	\revision{where:}
	\begin{itemize}[leftmargin=*]
		\item $ \delayDisc{\textit{pred}} \doteq \delayDiscTot{1}-1+\delayFusDiscTot $ is the length of the multi-step prediction;
		\item $ \delayDiscTot{\cardS} - \delayDiscTot{1} $ is the time between oldest and newest processed data;
		\item $ P_{\infty}(\delaySetDisc{}{}) $ solves the ARE where all active sensors are considered:
				\begin{equation}\label{eq:ARE}
				P_{\infty}(\delaySetDisc{}{}) = \iteration\left(P_{\infty}(\delaySetDisc{}{}),\delaySetDisc{\delayDiscTot{\cardS}}{}\right)
				\end{equation}
	\end{itemize}
\end{thm}

\begin{proof}
	See Appendix~\ref{app:proofThmCostFunction} in the supplementary material.
\end{proof}

\cref{alg:computecovariance} implements~\eqref{cost-function-open-loop}: line~\ref{alg-computeVar-ARE} solves the ARE~\eqref{eq:ARE}; loop~\ref{alg-computeVar-meas-begin}--\ref{alg-computeVar-meas-end} computes the multi-step KF iterations, with $ j $-th iteration involving the subset $ \{1,..., \cardS-j\} $; line~\ref{alg-computeVar-pred} computes the multi-step prediction.
This procedure yields the \revision{desired steady-state expected covariance $ \Psteady(\delaySetDisc{}{}) $}.

\begin{algorithm}
	\caption{\texttt{computeCovariance} subroutine\label{alg:computecovariance}}
	\begin{algorithmic}[1]
		\Require System $ (A,Q) $, state-output matrix $C_i$, noise covariance $ R_i(\cdot) $, communication and fusion delays $ \delayCommDisc $, $ \delayFusDisc $ for each active sensor $ i \in\mathcal{S}$, preprocessing delays $ \delaySetDisc{}{} $.
		\Ensure \revision{Expected} error covariance $ \Psteady(\delaySetDisc{}{}) $.
		\State Compute solution $ P_{\infty}(\delaySetDisc{}{}) $ of ARE with all sensors\label{alg-computeVar-ARE};
		\State $ P \leftarrow P_{\infty}(\delaySetDisc{}{}) $;
		\For{sensor amount $ i\leftarrow \cardS-1 $ \textbf{down to} $ 1 $}\label{alg-computeVar-meas-begin}
		\State multi-step KF iteration: $ P \leftarrow \iteration^{\delayDiscTot{i+1}-\delayDiscTot{i}}\left( P,\delaySetDisc{\delayDiscTot{i}}{}\right)$;
		\EndFor\label{alg-computeVar-meas-end}
		\State multi-step prediction: $ \Psteady(\delaySetDisc{}{}) \leftarrow \timeUp^{\delayDisc{\textit{pred}}}(P) $\label{alg-computeVar-pred};
		\State \textbf{return} $ \Psteady(\delaySetDisc{}{})$.
	\end{algorithmic}
\end{algorithm}

\revision{
	\autoref{fig:covarianceComputation} illustrates the procedure with $ \cardS = 3 $ active sensors.
	At time $k - \delayFusDiscTot$, the fusion station initiates the computation to produce an estimate at time $k$.
	When performing the fusion, the station 
	has access to the data from all sensors collected  until time $ k-\delayFusDiscTot-\delayDiscTot{3} $.
	However, due to the computation and communication delays, it will only have access to a subset of the sensor data after $ k-\delayFusDiscTot-\delayDiscTot{3} $. In particular, sensor 3 has the largest preprocessing-and-communication delay $ \delayDiscTot{3} $, 
	and the data it collects after time $ k-\delayFusDiscTot-\delayDiscTot{3} $ will only be received at the fusion station after time $k - \delayFusDiscTot$.
	The insight of~\cref{alg:computecovariance} is simple: the algorithm computes the expected covariance 
	until time $ k-\delayFusDiscTot-\delayDiscTot{3} $ (when all sensors are available), and then it collects the sporadic measurements 
	collected after $ k-\delayFusDiscTot-\delayDiscTot{3} $ which arrived at the fusion station before time $k - \delayFusDiscTot$ (\ie the ones from sensor 1 and 2 in the figure). In particular, accounting for all measurements collected until time $ k-\delayFusDiscTot-\delayDiscTot{3} $
	leads to the steady-state expected error covariance $ P_\infty(\delaySetDisc{}{}) $ satisfying~\eqref{eq:ARE}, 
	while the subsequent measurements are captured by~\eqref{cost-function-open-loop}.
	In \autoref{fig:covarianceComputation}, sensor 3 only provides one measurement, and the following four estimates use sensors 1 and 2.
	Afterwards, also sensor 2 becomes outdated and the last measurement updates only involve sensor 1.
	After the processed dataset has been used, the current-state estimate is retrieved with an open-loop prediction	compensating for the remaining delay (in this case induced by sensor 1 and fusion). 
	In the figure, the sensor contribution to the state estimates over time is highlighted with the matrix in the bottom row: at first, all sensor data are available (full bottom-left matrix), then some sensors become outdated, until no more sensor measurement is received and 
	the state estimate must be propagated in open loop
	(empty bottom-right matrix).
}

	\begin{rem}[Cost computation with state augmentation]
	\cref{cost-function-open-loop} can be equivalently written in a more compact way, by considering the augmented system with $ \delayDiscTot{\cardS} + \delayFusDiscTot $ consecutive states and $ C $ and $ R $ having nonzero blocks according to processed data. The \revision{cost $ \tr{\Psteady\left(\delaySetDisc{}{}\right)} $} would be retrieved by computing the steady-state \revision{expected} covariance of the augmented-state-estimate error, and cropping the bottom-right block. However, this is numerically inefficient and does not exploit the specific structure of the problem.
\end{rem}

\subsection{\revision{Extensions and Example}}

\revision{In this section, we discuss extensions of~\cref{alg:computecovariance} to an MPC-like setup and to the case of a 
multi-rate network. We conclude the section with a numerical example.}

\begin{rem}[Adaptive selection]\label{rem:timeVaryingAlg}
	\revision{
		\cref{alg:computecovariance} can be modified for an adaptive design, \eg 
		to deal with the case where the system parameters change overtime, or where we desire to schedule different sensors over time. 
		The time-varying counterpart of~\cref{prob:time-invariant-P-opt} might be solved multiple times over suitable horizons, in an MPC-like fashion.
		In particular, line~\ref{alg-computeVar-ARE} can be substituted with online KF iterations~\eqref{eq:KalmanUpdate}
		from time $ k_0 $ to $ k-\delayFusDiscTot-\delayDiscTot{\cardS} $, where all sensors are considered.	
	}
\end{rem}

\begin{figure}
	\centering
	\includegraphics{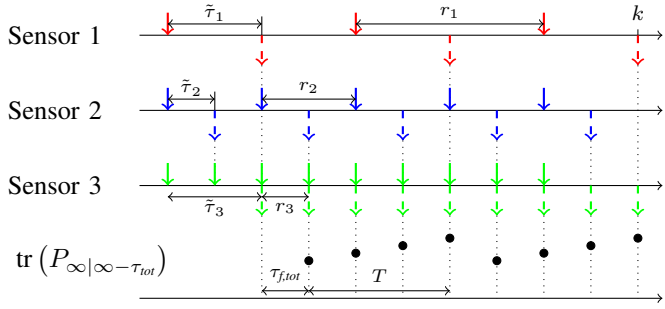}
	\caption{\revision{$ T $-periodic} cost with three multi-rate sensors. Solid arrows: sensor acquisitions; dashed arrows: data reception at central station.}
	\label{fig:periodicCovariance}
\end{figure}

\begin{rem}[Multi-rate networks]
	\cref{alg:computecovariance} can deal with networks where nodes have heterogeneous acquisition times $ r_i $ (see~\autoref{fig:sensorDelaysMultirate}).
	This fact is quite natural in processing networks: 
	for example, the drones in~\autoref{fig:vehicle-tracking} can use cameras with different frame rates.
	\revision{The corresponding information matrix (for the $ i $-th sensor) can be easily modeled in our setup as a 
	time-varying matrix:}
	\begin{equation}\label{eq:info-matrix-multi-rate} 
		\Gamma_i(\delayDisc{i},k) \doteq \begin{cases}
		C_i^T(R_i(\delayDisc{i}))^{-1}C_i, & \mbox{if } k_0=k \mod r_i\\
		0_{n\times n}, & \mbox{otherwise}
		\end{cases}
	\end{equation}
	\revision{which can be readily used in~\cref{alg:computecovariance}.}
	\autoref{fig:periodicCovariance} shows a qualitative behavior of the cost with three multi-rate sensors.
	
\end{rem}

\begin{figure}
	\centering
	\includegraphics[width=0.7\linewidth]{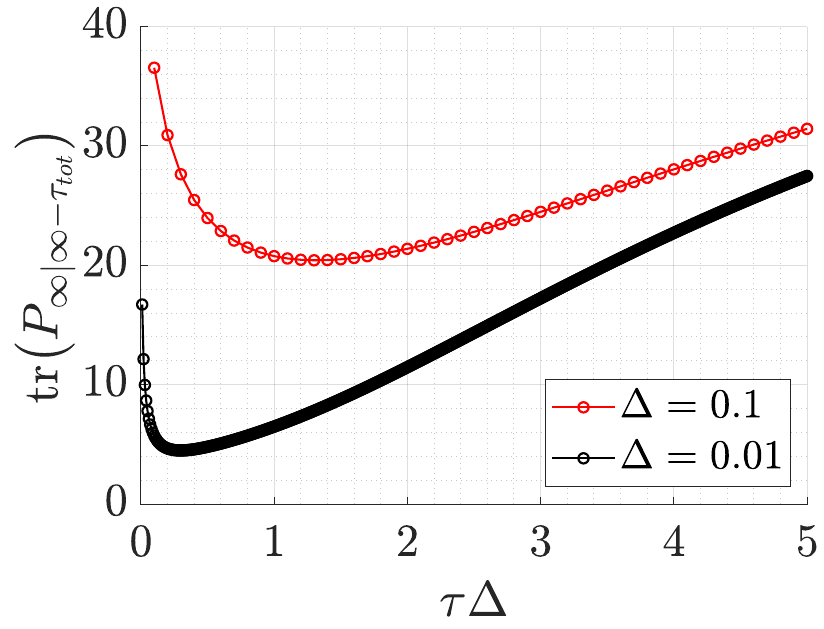}
	\caption{Cost function for a homogeneous network. The original continuous-time system has poles $\sigma(A)=\{-1,-0.1\}$ and $ Q=10I_2 $.}
	\label{fig:Psteady-disc-time-multidim-hom-net}
\end{figure}

\begin{ex}[Homogeneous network]\label{sec:hom-net-disc-time-scalar}
	Now that we are able to compute the cost function in the discrete-time case, 
	we test the simplified case of a homogeneous network where we only need to decide a single preprocessing delay for all sensors,
	\revision{without addressing sensor selection}.
	This allows establishing a connection with the formal setup in~\autoref{sec:hom-net-cont-time}.
	The numerical simulations in \autoref{fig:Psteady-disc-time-multidim-hom-net} exhibit a quasi-convex behavior, similarly to the continuous-time counterparts in~\autoref{fig:poftautwoparts}.
	This motivates us to explore greedy ``descend'' methods that attempt to iteratively minimize the cost function.\footnote{In the 
	homogeneous case, we only need to compute a single delay and the optimal preprocessing delay can be computed easily even by a brute-force search. However, our goal is to tackle the general case of heterogeneous networks where the problem becomes high-dimensional.}
\end{ex}

\section{A Greedy Algorithm for Selection}\label{sec:greedy-algorithms}

Since~\cref{prob:time-invariant-P-opt} cannot be solved analytically, we propose a two-step greedy approach
relying on the insights highlighted in~\cref{rem:insights}.
On the one hand, motivated by the need of choosing an optimal sensor subset,
the algorithm selects iteratively one sensor at a time, until the cost cannot be further decreased.
On the other hand, leveraging the intuition of a quasi-convex cost,
the delays of each tentative subset are optimized by a dedicated descent subroutine.
\autoref{sec:comp-wise-descent} shows how the latter optimizes the delays for a given sensor set,
while~\autoref{sec:sensor-selection} presents the full procedure.

\subsection{Sensor-wise Descent for Delay Selection}\label{sec:comp-wise-descent}

We propose a \emph{sensor-wise descent} algorithm to compute near-optimal computational delays $ \delaySetDisc{}{*} $ for a given active set $ \mathcal{S} $: 
we optimize one delay $ \delayDisc{i} $  at a time, by minimizing the associated one-dimensional problem, with all other delays being fixed.
\cref{alg:comp-wise-descent} shows the subroutine steps.
\revision{
	The to-be-returned delays and cost are initialized with the input delays $ \delaySetDisc{I}{} $ and the trace of the expected error covariance computed with $ \delaySetDisc{I}{} $, respectively (lines~\ref{cmp-wise-init-begin}--\ref{cmp-wise-init-end}).\footnote{The delays $ \delaySetDisc{I}{} $ are provided suitably by the algorithm in~\autoref{sec:sensor-selection}.}
	The outer loop between lines~\ref{cmp-wise-loop-sensors-begin}--\ref{cmp-wise-loop-sensors-end} optimizes the delay $ \delaySetDisc{}{*}[i] =\delayDisc{i} $ with a one-dimensional descent.
	For each delay, an explorative iteration is first used to set the sign of the unitary stepsize $ \alpha $, according to the descent direction (line~\ref{cmp-wise-set-descent-1}--\ref{cmp-wise-set-descent-2}).
	Given $\alpha$, the inner loop (lines~\ref{cmp-wise-min-cost-update}--\ref{cmp-wise-cost-update}) computes the near-optimal delay: 
	the descent direction is explored until a local minimum is found or surpassed (condition~\ref{cmp-wise-desc-cond}).
	The best achieved delay is restored and saved in line~\ref{cmp-wise-final-update}.
}
\revision{\autoref{fig:three-sensors-function} shows the cost function with {three sensors}: it looks quasi-convex, consistently with the homogeneous networks.
\autoref{fig:three-sensors-descent} illustrates an execution of the proposed algorithm and its sensor-wise descent nature.}

\begin{algorithm}
	\caption{\texttt{sensorWiseDescent} subroutine\label{alg:comp-wise-descent}}
	\begin{algorithmic}[1]
		\Require System $ (A,Q) $, state-output matrix $C_i$, noise covariance $ R_i(\cdot) $, communication and fusion delays $ \delayCommDisc $, $ \delayFusDisc $ for each active sensor $ i \in\mathcal{S}$, initial delays $ \delaySetDisc{I}{} $.
		\Ensure Near-optimal delay set $ \delaySetDisc{}{*}$, cost $ \tr{\Psteady(\delaySetDisc{}{*})} $.
		\State $ p_{\textit{min}}\leftarrow \tr{\texttt{computeCovariance}(A,Q,\mathcal{S},\delaySetDisc{I}{})} $\label{cmp-wise-init-begin};
		\State $ \delaySetDisc{}{*} \leftarrow \delaySetDisc{I}{} $\label{cmp-wise-init-end};
		\For{\textbf{each} $ i\in \mathcal{S} $}\label{cmp-wise-loop-sensors-begin}
		\State Stepsize $ \alpha \leftarrow -1 $ ;  \label{cmp-wise-set-descent-1}\algorithmiccomment{default: delay $ \delayDisc{i} $ is decreased}
		\State$ \delaySetDisc{}{*}[i] \leftarrow \delaySetDisc{}{*}[i]+\alpha$;\label{cmp-wise-set-descent-1/2}
		\State$ p_{\textit{curr}}\leftarrow \tr{\texttt{computeCovariance}(A,Q,\mathcal{S},\delaySetDisc{}{*})} $;
		\If{$ p_{\textit{min}} \le p_{\textit{curr}} $}
		\State $ \alpha \leftarrow +1 $; \algorithmiccomment{delay $ \delayDisc{i} $ is increased}
		\EndIf\label{cmp-wise-set-descent-2}
		\Do\label{cmp-wise-loop-delay-begin}
		\State$ p_{\textit{min}} \leftarrow p_{\textit{curr}} $\label{cmp-wise-min-cost-update};
		\State$ \delaySetDisc{}{*}[i] \leftarrow \delaySetDisc{}{*}[i]+\alpha$;
		\State$ p_{\textit{curr}}\leftarrow \tr{\texttt{computeCovariance}(A,Q,\mathcal{S},\delaySetDisc{}{*})} $\label{cmp-wise-cost-update};
		\doWhile{$ {p_{\textit{min}}}\le {p_{\textit{curr}}} $\label{cmp-wise-desc-cond}}
		\State$ \delaySetDisc{}{*}[i] \leftarrow \delaySetDisc{}{*}[i]-\alpha$;\label{cmp-wise-final-update}
		\EndFor\label{cmp-wise-loop-sensors-end}
		\State \textbf{return} $\delaySetDisc{}{*}, \ p_{\textit{min}}$.
	\end{algorithmic}
\end{algorithm}

\begin{figure}
	\centering
	\includegraphics[width=.7\linewidth]{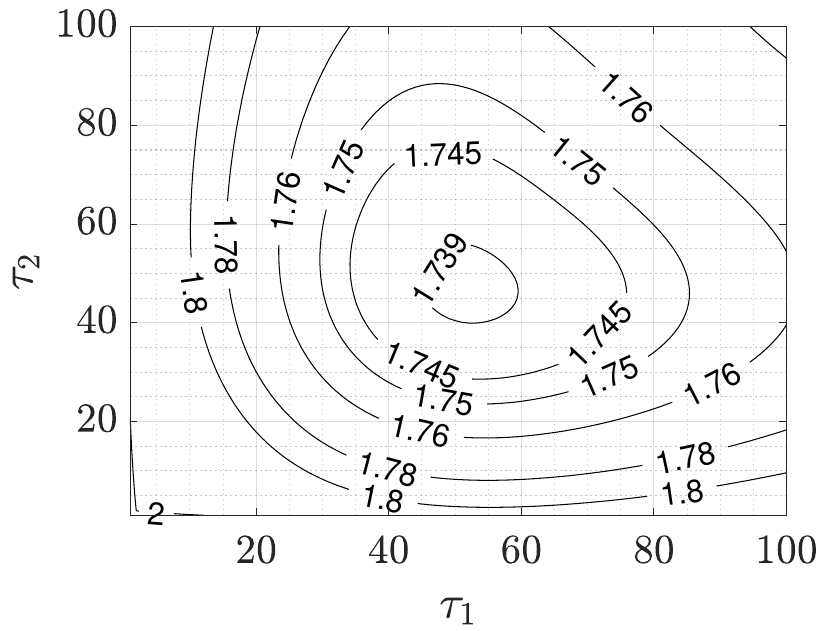}
	\caption{Cost-function levels with constant $ \delayCommDisc $, $ \delayFusDisc $ ($ \delayDisc{3} $ is fixed).}
	\label{fig:three-sensors-function}
\end{figure}
\begin{figure}
	\centering
	\includegraphics[width=.7\linewidth]{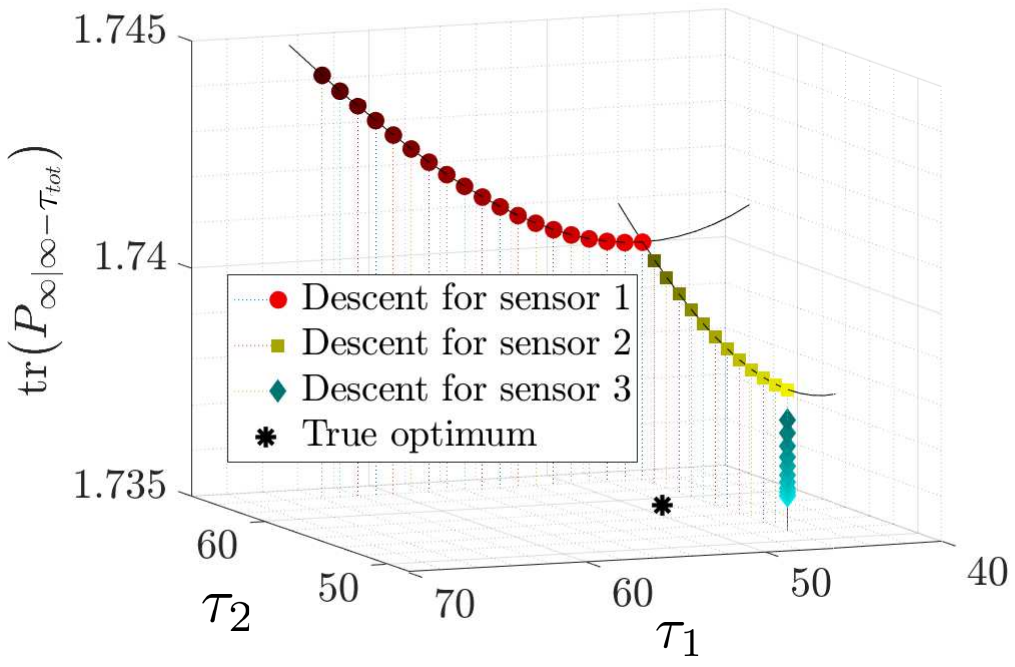}
	\caption{Visualization of \texttt{sensorWiseDescent} on cost function in~\autoref{fig:three-sensors-function} ($ \delayDisc{3} \mbox{ goes from } 64 \mbox{ to }52 $): iterations go from darker to lighter marks.}
	\label{fig:three-sensors-descent}
\end{figure}

\subsection{Sensor Selection}\label{sec:sensor-selection}

We now present the main procedure to solve~\cref{prob:time-invariant-P-opt}.
A greedy algorithm selects one sensor at a time, as long as the cost can be dropped.
For each tentative subset, preprocessing delays are optimized as described in~\autoref{sec:comp-wise-descent}.

\begin{algorithm}
	\caption{Sensor selection\label{alg:add}}
	\begin{algorithmic}[1]
		\Require System $ (A,Q) $, state-output matrix $C_i$, noise covariance $ R_i(\cdot) $, communication and fusion delays $ \delayCommDisc $, $ \delayFusDisc $ for each available sensor $ i \in\sensSet $.
		\Ensure Near-optimal sensor set $ \mathcal{S}^* $ and delay set $\delaySetDisc{}{*}$.
		
		\State Compute optimal delays $ \delayDisc{\textit{opt,i}}^\textit{o} $ for one-sensor subsets $ \{i\} $\label{sens-sel-single-sens-opt};	
		\State $\mathcal{S}^* \leftarrow $ one-sensor subset achieving minimum cost\label{sens-sel-initialization-begin};
		\State $\delaySetDisc{}{*} \leftarrow $ \{optimal delay $ \delayDisc{\textit{opt},\mathcal{S}^*}^\textit{o} $ for $ \mathcal{S}^* $\};
		\State $ p_{\textit{min}} \leftarrow $ minimum cost achieved by $ \mathcal{S}^* $\label{sens-sel-initialization-end};
		\Do\label{sens-sel-add-sens-loop-begin}
			\State toSelect $ \leftarrow \emptyset$;
			\For{\textbf{each} {\normalfont sensor toTry} $  \in \sensSet\backslash\mathcal{S}^* $}\label{sens-sel-best-sens-loop-begin}
				\State $\mathcal{S}_{\textit{curr}} \leftarrow \mathcal{S}^*\cup\{\mbox{toTry}\}$\label{sens-sel-temp-set};
				\State $\delaySetDisc{\textit{curr}}{} \leftarrow \delaySetDisc{}{*}\cup\left\lbrace \delayDisc{\textit{opt},\mbox{\scriptsize toTry}}^\textit{o} \right\rbrace$\label{sens-sel-delay-init};
				\State $[\delaySetDisc{\textit{curr}}{}, p_{\textit{curr}}]\leftarrow$ 
				\Statex \hfill $\texttt{sensorWiseDescent}\left(A,Q,\mathcal{S}_{\textit{curr}},\delaySetDisc{\textit{curr}}{}\right)$\label{sens-sel-cmp-wise-descent};
				\If{${p_{\textit{min}}}> {p_{\textit{curr}}}$}\label{sens-sel-temp-best-begin}
					\State toSelect $ \leftarrow $ toTry;\label{sens-sel-to-add}
					\State $p_{\textit{min}}\leftarrow p_{\textit{curr}}$;
					\State $\delaySetDisc{\textit{curr}}{*} \leftarrow \delaySetDisc{\textit{curr}}{}$;\label{sens-sel-current-best-delays}
				\EndIf
			\EndFor\label{sens-sel-best-sens-loop-end}
			\If{ $ \exists $ {\normalfont toSelect}}\label{sens-sel-sens-set-update-begin}
				\State $ \mathcal{S}^*\leftarrow\mathcal{S}^*\cup\{\mbox{toSelect}\} $;
				\State $\delaySetDisc{}{*} \leftarrow \delaySetDisc{\textit{curr}}{*} $;   
			\EndIf\label{sens-sel-sens-set-update-end}
		\doWhile{${p_{\textit{min}}}\le {p_{\textit{curr}}}$ \textbf{or} $\cardS= \lvert \sensSet \rvert $\label{sens-sel-add-sens-loop-end}}
		\State \textbf{return} $ \mathcal{S}^*, \ \delaySetDisc{}{*}$.
	\end{algorithmic}
\end{algorithm}

\cref{alg:add} shows the pseudocode of the proposed algorithm.
\revision{
	First, the optimal performance is computed for each available sensor, by taking one at a time
	(line \ref{sens-sel-single-sens-opt}).
	The to-be-returned sensor and delay subsets $\mathcal{S}^*$, $\delaySetDisc{}{*} $ and the minimum cost $ p_{\textit{min}} $ are initialized with the sensor achieving the minimum cost (lines~\ref{sens-sel-initialization-begin}--\ref{sens-sel-initialization-end}).
	The outer loop~\ref{sens-sel-add-sens-loop-begin}--\ref{sens-sel-add-sens-loop-end} adds one sensor at a time to the selection $ \mathcal{S}^* $, stopping when the cost hits a local minimum (no other sensor can be added to further reduce the cost) or all available sensors have been selected.
	The inner loop~\ref{sens-sel-best-sens-loop-begin}--\ref{sens-sel-best-sens-loop-end}, given the current selection $ \mathcal{S}^* $, builds the tentative subsets $ \mathcal{S}_{\textit{curr}} $ (line~\ref{sens-sel-temp-set}) by adding the so-far-excluded sensors (\emph{toTry}) one at a time.
	The near-optimal delays $ \delaySetDisc{\textit{curr}}{} $ for the tentative set are initialized with the best delays obtained so far for the sensors in $ \mathcal{S}_{\textit{curr}} $ (line~\ref{sens-sel-delay-init}), \ie with the current near-optimal delays $ \delaySetDisc{}{*} $ for the already-selected sensors, and with the single-sensor optimal delay for the tentative sensor: intuitively, a ``small'' difference between subsets yields ``small'' differences between optimal delays. The sensor-wise descent is in charge of computing the near-optimal delays and cost for each subset (line~\ref{sens-sel-cmp-wise-descent}).
	When a tentative subset hits a new minimum (line~\ref{sens-sel-temp-best-begin}), the sensor \emph{toTry} becomes \emph{toSelect} (line~\ref{sens-sel-to-add}), \ie it is the best candidate to be added to the selected subset.
	The temporary variable $ \delaySetDisc{\textit{curr}}{*} $ allows not to overwrite the delays used to initialize \texttt{sensorWiseDescent}.
	When all available sensors have been tried, the one \emph{toSelect} (if any) and the new near-optimal delays are stored in the to-be-returned variables (lines~\ref{sens-sel-sens-set-update-begin}--\ref{sens-sel-sens-set-update-end}).
}

\begin{rem}[Non-detectable subsystems]
	\revision{
		Some costs in line~\ref{sens-sel-single-sens-opt} may not be computable if pairs $ (A,C_i) $ are not detectable.
		If this only holds for some sensors, the others may be involved in the initialization.
		Otherwise, the latter may be replaced by the greedy selection of the minimum-cardinality, minimum-cost sensor subset providing detectability.
	}
\end{rem}

\begin{rem}[Finite channel capacity]
	\revision{
		The channel capacity can be handled by adding a termination condition in line~\ref{sens-sel-add-sens-loop-end},
		stopping the algorithm when the selected sensors ``fill" the channel.
		A threshold $ \bar{\cardS} < \lvert \sensSet \rvert $ may limit the maximum number of selected sensors $ \cardS \le \bar{\cardS} $ (in this case, the structure of the algorithm ensures that the best $ \bar{\cardS} $-sensor subset is selected),
		or it may be computed online by considering the bandwidth utilization of each selected sensor.
		It is worth noting that several work in the literature deal with channel capacity, usually from a scheduling perspective, \eg~\cite{2018arXiv180405618W,10.1145/3209582.3209589,5281763}.
		Merging these two approaches may be possible, but goes beyond the scope of this paper, which presents a standalone, computation-efficient strategy.
		Also, the algorithm might be run multiple times over a suitable time horizon if a scheduling-like
		design is needed (cf.~\cref{rem:timeVaryingAlg}).
	}
\end{rem}

\begin{rem}[User-driven selection]
	\revision{
		If the task needs specific sensor data (as images or infra-red),~\cref{alg:add} can be customized, \eg
		having the initialization of the selected subset (lines~\ref{sens-sel-initialization-begin}--\ref{sens-sel-initialization-end})
		forced to include the corresponding sensors.
	}
\end{rem}

\section{Numerical Simulations}\label{sec:simulations}

Inspired by~\autoref{fig:vehicle-tracking}, we simulate an heterogeneous sensor network in charge of tracking position and velocity of a ground vehicle.
The system state $ \boldsymbol{x} = [\mathit{x} \, \mathit{\dot{x}} \, \mathit{y} \, \mathit{\dot{y}}]^T $, $ \mathit{x}$ and $ \mathit{y} $ being spatial coordinates,
has dynamics given by~\eqref{eq:processModel} with
\begin{equation}\label{simulationSystem}
	A = \begin{bmatrix}
		1 & \Delta & 0 & 0 \\
		0 & 1 & 0 & 0 \\
		0 & 0 & 1 & \Delta \\
		0 & 0 & 0 & 1
	\end{bmatrix} 
	\
	Q = \begin{bmatrix}
		\sigma^2_\mathit{x} & \\
		& \sigma^2_\mathit{y}
	\end{bmatrix} \otimes 
	\begin{bmatrix}
		\nicefrac{\Delta^2}{4} & \nicefrac{\Delta^3}{2} \\
		\nicefrac{\Delta^3}{2} & \Delta^2
	\end{bmatrix} 
\end{equation}
where \revision{$ \sigma_\mathit{x}^2=\sigma_\mathit{y}^2=0.1 $ convey the inaccuracy given by approximating the actual vehicle motion with constant speed, and we set the sampling time $ \Delta = 1\si{\milli\second}$}.\\
The available set $ \sensSet $ is composed of six smart sensors:
\begin{itemize}[leftmargin=*]
	\item sensor 1 models a drone equipped with a powerful GPU-CPU processing hardware and a high-resolution camera working at 60fps, with a sparse matrix $ C_i \in\Real{4\times 100} $ with density coefficient 0.3 (command \texttt{sprand} in Matlab);
	\item sensors 2 and 3 model drones with low-resolution cameras working at 30fps, with sparse matrices $ C_i \in\Real{4\times 20} $;
	\item sensors 4, 5 and 6 model event cameras collecting events at 100Hz, which output noisy estimates of the position.
\end{itemize}

\begin{table}
	\centering
	\caption{Parameters used in simulation for each sensor (sensor ID in parenthesis).}
	\begin{tabular}{|l|c|c|c|c|}
		\hline 
		\textbf{Parameter} & \boldmath $b_i$ & \boldmath  $c_i $  & \boldmath $f_i$ & \boldmath $r_i$ \\
		\hline
		Powerful drone (1) & \si{\num{5E-3}} & 2500 & 1250 & \si{\num{15}} \\ 
		\hline
		Lightweight drone (2, 3) & \si{\num{7.5E-2}} & 110 & 55 & \si{\num{30}} \\ 
		\hline
		Event camera (4, 5, 6) & \si{\num{3.4}} & 2 & 1 & \si{\num{10}} \\ 
		\hline
	\end{tabular}
	\label{tab:simulationsParams}
\end{table}

\autoref{tab:simulationsParams} collects the parameters used in the simulation, assuming WiFi connection at 25 Mbps, needed to ensure real-time performance of the high-resolution camera.
\revision{
	Also, we assume a packet-loss probability $ 1-\lambda = 0.25 $ for all sensors.
	We based our choice of parameters on the real-world experiments described in~\cite{4629874,2019arXiv190408405G},
	and assume varying communication and fusion delays according to model~\eqref{comm-fus-del-var}.
	For instance, drone 2 is assumed equipped with modest computational capabilities: with a 30\si{\milli\second}-long preprocessing delay, this sensor can estimate the vehicle position with an error standard deviation of 5\si{\centi\meter}.
	On the other hand, its low-quality camera provides images which take little time to be compressed/delivered via wireless.
	For instance, transmitting raw images (\ie with minimum preprocessing $ \delayDisc{2} = 1\si{\milli\second} $) takes 110\si{\milli\second}.
	We also assumed modest resources at the fusion center, which takes half of the communication time to process data from each sensor.
}
The preprocessing delays range from 10\si{\milli\second} to 290\si{\milli\second}, only considering multiples of 20 to make the brute-force search feasible. 

\begin{figure}
		\includegraphics[width=0.9\linewidth]{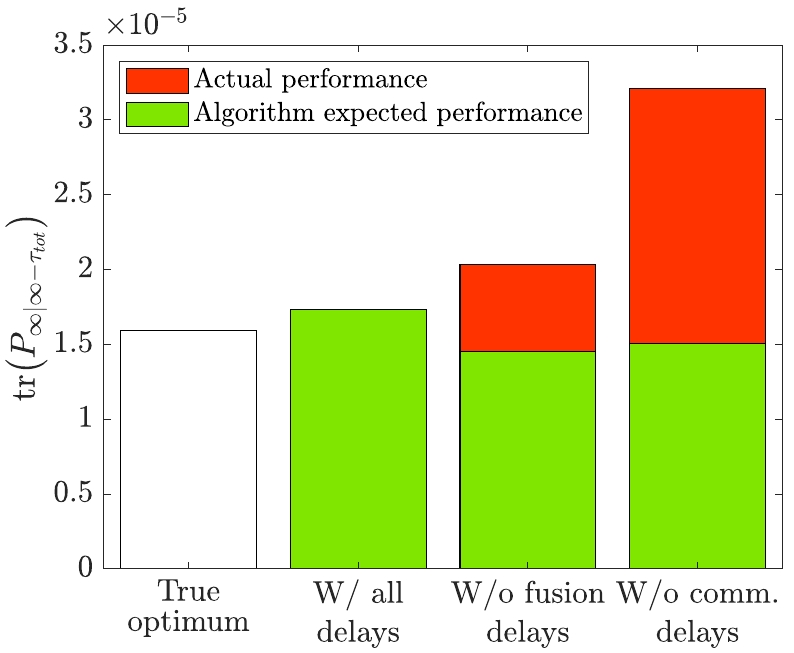}
		\caption{Optimal and selection algorithm costs with different models.
			Numerical values (actual performance): $ \{1.84, 1.94, 2.17, 3.32\}\!\times\!10^{-5} $.}
		\label{fig:algs-vs-true}
\end{figure}

\autoref{fig:algs-vs-true} shows the optimal cost (white bar) and the ones achieved by the algorithm considering either the full model or only a subset of the delays. 
\revision{In particular, red bars show the actual cost,
while green bars represent the performance predicted by the algorithm, which is an underestimate when the used model lacks either communication or fusion delays.}
The optimal cost is computed via a brute-force search and is only used for benchmarking, since this strategy does not scale in the size of the network.
\revision{The greedy selection makes an error of about 5.4\% when all delays are considered, of 18\% when neglecting fusion delays, and of 80\% when neglecting communication delays.}
This translates into a larger tracking error: for instance, using only sensors 1 raises to 4\si{\meter/\second} the optimal error on velocity, which is 3\si{\meter/\second}.
Also, selecting more sensors than necessary may impact other aspects, such as energy consumption.

\begin{table}
	\centering
	\caption{Sensors and delays: optimal and greedy selection.}
	{\small 
	\begin{tabular}{|l|l|c|c|c|c|c|c|}
		\hline
		\multicolumn{2}{|l|}{\textbf{Sensor}} & \textbf{1} & \textbf{2} & \textbf{3} & \textbf{4} & \textbf{5} & \textbf{6} \\
		\hline
		\multicolumn{2}{|l|}{True optimum} & $ \diagup $ & 30\si{\milli\second} & 30\si{\milli\second} & $ \diagup $ & $ \diagup $ & $ \diagup $ \\
		\hline
		\multirow{3}{*}{Greedy} & All delays & $ \diagup $ & 50\si{\milli\second} & 50\si{\milli\second} & $ \diagup $ & $ \diagup $ & $ \diagup $ \\
		\cline{2-8}
		& W/o fusion & 50\si{\milli\second} & 30\si{\milli\second} & 30\si{\milli\second} & $ \diagup $ & $ \diagup $ & $ \diagup $ \\
		\cline{2-8}
		& W/o comm. & 30\si{\milli\second} &  $ \diagup $ &  $ \diagup $ &  $ \diagup $ &  $ \diagup $ &  $ \diagup $ \\
		\hline
	\end{tabular}}
	\label{tab:delays}
\end{table}

Considering all delays, the proposed algorithm selects both the optimal sensor subset and near-optimal preprocessing delays (\autoref{tab:delays}).
\revision{Notice that, according to the intuition, the optimal choice features the same delay for both the selected sensor (sensors 2 and 3), being these (almost) identical.}
The powerful drone \revision{(sensor 1)} is discarded because of its heavy impact on communication and fusion latency: this also explains the performance drop when communication delays are not considered, as this sensor --when neglecting communication delays-- 
erroneously appears as the best sensor to choose.
The event cameras are excluded because of their large preprocessing noise, not balanced by their fast acquisition rate and small communication and fusion delays.

\begin{figure}
	\centering
	\includegraphics[width=0.9\linewidth]{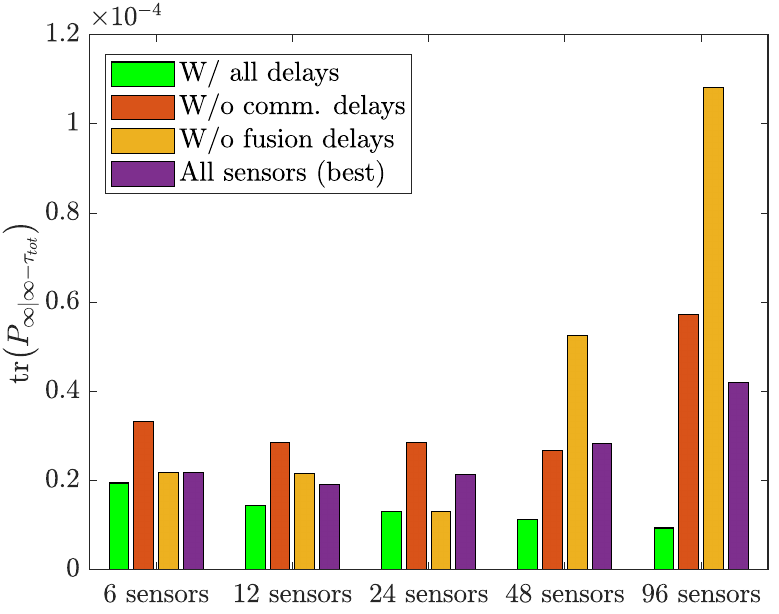}
	\caption{Greedy-selection and all-sensor cost with increasing set size.}
	\label{fig:alg-multi-sensors}
\end{figure}

\begin{table}
	\centering	
	\caption{Number of available sensors $ \lvert \sensSet \rvert $, and number of sensors $ \lvert \mathcal{S} \rvert$ selected by the proposed greedy algorithm.}
	{\small 
	\begin{tabular}{|l|c|c|c|c|c|}  
		\hline
		& \multicolumn{5}{c|}{\textbf{Sensor amount}} \\
		\hline
		Available & 6 & 12 & 24 & 48 & 96 \\
		\hline
		Greedy w/ all delays & 2 & 2 & 3 & 5 & 5 \\
		\hline
		Greedy w/o fusion delays & 3 & 6 & 3 & 48 & 96 \\
		\hline
	\end{tabular}}
	\label{tab:sensor-amount}
\end{table}

\revision{If we consider a larger number of available sensors, accounting for the fusion delays becomes even more important.
\autoref{fig:alg-multi-sensors} shows the greedy performance with increasing set size, together with the minimum cost obtained with all sensors (over 20 iterations of \texttt{sensorWiseDescent} with random initial delays).
For each increment in set size, the newly added sensors have different parameters either for communication delay (coefficient $ c_i $) or for preprocessing and fusion delays (coefficients $ b_i $ and $ f_i $), which range from 0.9 to 0.1 of the original ones in~\autoref{tab:simulationsParams}. One may think about these variations as different choices for the sensor hardware or better channel state/device position.
We see the impact of better-performing sensors in the sets with 12 and 24 sensors, where the costs obtained when neglecting some delays are in the range of the ones  in~\autoref{fig:algs-vs-true}. However, from 6 to 96 sensors, the gaps between the proposed approach (green bars) and alternatives that do not account for communication and fusion delays steadily increases, which is particularly evident with 48 and 96 sensors.
Also, given the availability of new better-performing sensors when increasing the sensor set, the near-optimal costs decrease.
\autoref{tab:sensor-amount} reports the number of sensors selected by~\cref{alg:add}: 
we can see that neglecting fusion delays leads in general to the choice of an unnecessarily large set of sensors, while the proposed approach ensures enhanced performance (as highlighted by the previous figures) while using a smaller set of sensors.}

\section{Conclusions and Future Work}\label{sec:conclusions}

In this paper, we investigate optimal estimation in a processing network in the presence of communication and computational delays.
Anytime sensor preprocessing is modeled with noise intensity decreasing with the amount of computation.
For homogeneous networks  monitoring a continuous-time scalar system we prove that the preprocessing delay can be optimized analytically, and observe that performing no preprocessing is typically a suboptimal policy.
We then consider the general case of heterogeneous networks and discuss the joint problem of computing the 
optimal amount of preprocessing at the sensors and selecting the most informative sensors. 
We develop a greedy algorithm for near-optimal delay-and-sensor selection. 
Numerical simulations show the effectiveness of such algorithm, and confirm that the proposed model 
leads to more accurate estimates.

There are several avenues for future work.
First, it is desirable to obtain suboptimality bounds on the proposed algorithm.
Second, the model can be made more realistic by
introducing non-ideal communication (unreliability, random delays), nonlinear dynamics, or parameter uncertainty.
Finally, it would be interesting to consider a fully distributed setup, where the estimation process is solved by local 
exchange of information, without a central fusion station. 

\ifCLASSOPTIONcompsoc
\section*{Acknowledgments}
\else
\section*{Acknowledgment}
\fi

This work was partially funded by the ONR RAIDER program (N00014-18-1-2828), the CARIPARO Foundation Visiting Programme ``HiPeR", and the Italian Ministry of Education PRIN n. 2017NS9FEY and under the initiative ``Departments of Excellence" (Law 232/2016).


\begin{IEEEbiography}[{\includegraphics[width=1in,height=1.25in,clip,keepaspectratio]{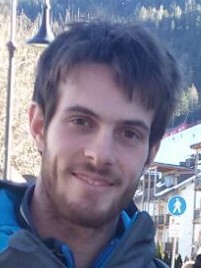}}]
	{Luca Ballotta} received his Bachelor's Degree in Information Engineering in 2017 and his Master's Degree in Automation Engineering in 2019 from the Universityof Padova, where he is currently pursuing the Ph.D. degree in Information Engineering, curriculum of Information science and technology. His research interests include multi-agent systems and distributed optimization for processing networks.
\end{IEEEbiography}
\vskip -2\baselineskip plus -1fil
\begin{IEEEbiography}[{\includegraphics[width=1in,height=1.25in,clip,keepaspectratio]{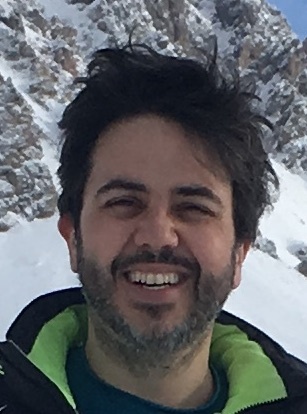}}]
	{Luca Schenato} received the Dr. Eng. degree in electrical engineering from the University of Padova in 1999 and the Ph.D. degree in Electrical Engineering and Computer Sciences from the UC Berkeley, in 2003. He held a post-doctoral position in 2004 and a visiting professor position in 2013-2014 at U.C. Berkeley. Currently he is Associate Professor at the Information Engineering Department at the University of Padova. His interests include networked control systems, multi-agent systems, wireless sensor networks, smart grids and cooperative robotics. Luca Schenato has been awarded the 2004 Researchers Mobility Fellowship by the Italian Ministry of Education, University and Research (MIUR), the 2006 Eli Jury Award in U.C. Berkeley and the EUCA European Control Award in 2014, and IEEE Fellow in 2017. He served as Associate Editor for IEEE Trans. on Automatic Control from 2010 to 2014 and he is he is currently Senior Editor for IEEE Trans. on Control of Network Systems and Associate Editor for Automatica.
\end{IEEEbiography}
\vskip -2\baselineskip plus -1fil
\begin{IEEEbiography}[{\includegraphics[width=1in,height=1.25in,clip,keepaspectratio]{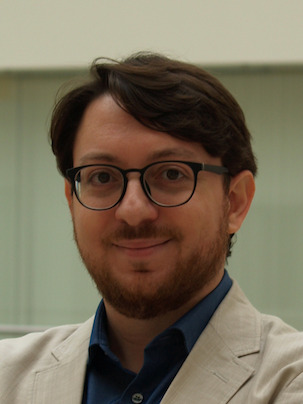}}]
	{Luca Carlone} is the \emph{Charles Stark Draper} Assistant Professor in the Department of Aeronautics and Astronautics at the Massachusetts Institute of Technology, and a Principal Investigator in the Laboratory for Information \& Decision Systems (LIDS). He has obtained a B.S. degree in mechatronics from the Polytechnic University of Turin, Italy, in 2006; an S.M. degree in mechatronics from the Polytechnic University of Turin, Italy, in 2008; an S.M. degree in automation engineering from the Polytechnic University of Milan, Italy, in 2008; and a Ph.D. degree in robotics also the Polytechnic University of Turin in 2012. He joined LIDS as a postdoctoral associate (2015) and later as a Research Scientist (2016), after spending two years as a postdoctoral fellow at the Georgia Institute of Technology (2013-2015). His research interests include nonlinear estimation, numerical and distributed optimization, and probabilistic inference, applied to sensing, perception, and decision-making in single and multi-robot systems. His work includes seminal results on certifiably correct algorithms for localization and mapping, as well as approaches for visual-inertial navigation and distributed mapping. He is a recipient of the 2017 Transactions on Robotics King-Sun Fu Memorial Best Paper Award, the Best Paper award at WAFR 2016, the Best Student Paper award at the 2018 Symposium on VLSI Circuits, and was best paper finalist at RSS 2015.
\end{IEEEbiography}
	\clearpage
	\setcounter{page}{1}
	\appendices

\section{Proof of Theorem~\autoref{thm-scalar-total}}\label{app:proof-thm-scalar-total}

The steady-state error variance for the outdated estimate $ \hat{x}_{t-\delayCont_\textit{tot}}(\delayCont) $ is the solution of the continuous-time ARE where all sensors are considered:
\begin{equation}
2ap_\infty(\delayCont) - \sigma^2_w + \frac{\delayCont}{\tilde{b}}p^2_\infty(\delayCont) = 0
\end{equation}	
An open-loop prediction of length $ \delayCont_{\textit{tot}} $ then computes the current-time estimate $\hat{x}_{t}(\delayCont)$.
The error associated with the prediction has dynamics
\begin{equation}
d\Tilde{x}_{s}(\delayCont) = a\Tilde{x}_{s}(\delayCont)ds + dw_{s}, \qquad t-\delayCont_\textit{tot}\le s\le t
\label{open-loop-error-dynamics-scalar}
\end{equation}
The error at time $ t $ is then given by integrating~\eqref{open-loop-error-dynamics-scalar} with initial condition $\Tilde{x}_{t-\delayCont_\textit{tot}}(\delayCont)$:
\begin{equation}
\Tilde{x}_{t}(\delayCont) = \mbox{e}^{a\delayCont_\textit{tot}}\Tilde{x}_{t-\delayCont_\textit{tot}}(\delayCont)+\Bar{w}(\delayCont_\textit{tot})
\end{equation}
where $\Bar{w}(\delayCont_\textit{tot})$ is the stochastic integral of $ w_s $ in the interval $[t-\delayCont_\textit{tot},\: t]$. The steady-state error variance is then 
\begin{align}
\begin{split}
\psteady(\delayCont) & \overset{(i)}{=} \mbox{var}(\mbox{e}^{a\delayCont_\textit{tot}}\Tilde{x}_{t-\delayCont_\textit{tot}}(\delayCont)) + \mbox{var}(\Bar{w}(\delayCont_\textit{tot})) =\\
& = \mbox{e}^{2a\delayCont_\textit{tot}}p_\infty(\delayCont)+\frac{\sigma^2_w}{2a}\left(\mbox{e}^{2a\delayCont_\textit{tot}}-1\right)
\end{split}
\end{align}
where $(i)$ is motivated by uncorrelated terms. Indeed, $\tilde{x}_{t-\delayCont_\textit{tot}} \in \mbox{span}\{x_{t_0}, w_s, v_s : t_0 \le s\le t-\delayCont_\textit{tot}\}$, while $\Bar{w}(\delayCont_\textit{tot}) \in \mbox{span}\{w_s : t-\delayCont_\textit{tot} \le s \le t\}$,
whose intersection has zero measure.\\
The variance $ \psteady(\delayCont) $ is quasi-convex with both constant and $\delayCont$-varying communication and fusion delays. This can be proved, e.g., with a graphical analysis. In virtue of both this fact and limits~\eqref{error-variance-limits-scalar}, the point of minimum $\delayCont_\textit{opt}$ exists unique and is strictly positive.\\
With constant delays $\delayCommCont$, $\delayFusCont$, standard computations show that $\delayCont_\textit{opt}$ must satisfy~\cref{3rd-degree-eq-scalar-general}.

\section{Alternative Preprocessing Models}\label{sec:other-functions}

We consider two alternative models to the measurement variance~\eqref{R-model-cont-time-hom-net}.
These involve a coefficient $\gamma$ that can be understood as the convergence rate of an anytime algorithm.\\
\begin{cor}[Non-ideal preprocessing]
	Given system~\eqref{eq:processModelCont}--\eqref{homogeneous-network} and hypotheses as per~\cref{thm-scalar-total} with 
	\begin{equation}
		\tilde{R}(\delayCont)=\frac{b}{\delayCont^{\gamma}}I_m \qquad \gamma> 0
	\end{equation}
	the steady-state error variance $\psteady(\delayCont)$ has a unique global minimum $\delayCont_\textit{opt}>0$.
\end{cor}
\begin{proof}
	It can be seen that limits (\ref{error-variance-limits-scalar}) still hold and $ \psteady(\delayCont) $ is strictly quasi-convex on $\Realp{}$ (\eg via graphical analysis) with both models~\eqref{comm-fus-del-const}--\eqref{comm-fus-del-var}.
\end{proof}
The second model comes into play with anytime algorithms with exponential convergence, as the ones shown in~\cite{10.1007/978-3-642-37213-1_16}. 
\begin{cor}[Exponential-convergence anytime algorithms]
	Given system~\eqref{eq:processModelCont}--\eqref{homogeneous-network} and hypotheses as per Theorem~\ref{thm-scalar-total} with 
	\begin{equation}
		\tilde{R}(\delayCont)=b\mbox{e}^{-\gamma\delayCont}I_m \qquad \gamma> 0
	\end{equation}
	the steady-state error variance $\psteady(\delayCont)$ has a unique global minimum $\delayCont_\textit{opt}>0$:
	\begin{enumerate}[leftmargin=*]
		\item with constant delays as per~\eqref{comm-fus-del-const}, if and only if 
		\begin{equation}
		\gamma > 2\sqrt{\frac{\sigma^2_w}{\tilde{b}}+a^2}  \label{eq:condition-gamma}
		\end{equation}
		\item with $\delayCont$-varying delays as per~\eqref{comm-fus-del-var}, always.
	\end{enumerate}
\end{cor}
\begin{proof}
	We address the two cases separately.
	\begin{enumerate}[leftmargin=*]
		\item With model~\eqref{comm-fus-del-const}, $\delayCont_\textit{opt}$ can be computed in closed form by setting $ \psteady'(\delayCont) = 0 $. This has the unique solution
		\begin{equation}
		\delayCont_\textit{opt} = \frac{1}{\gamma}\left[\ln\left(\frac{\gamma^2}{4}-a^2\right)+\ln\left(\frac{\tilde{b}}{\sigma^2_w}\right)\right]
		\end{equation}
		which is strictly positive if and only if~\eqref{eq:condition-gamma} holds.
		\item With model~\eqref{comm-fus-del-var}, $ \psteady(\delayCont)$ is quasi-convex (easily verifiable, e.g., via graphical analysis) for any $ \gamma $.
	\end{enumerate}
\end{proof}

\section{Proof of~\cref{prop-params}} \label{app:proof-prop-params}

For convenience, we recall the statement of the implicit function theorem, which is used in the proof.

\begin{thm}[Implicit function]\label{thm:Dini}
	Let F be a continuously differentiable function on some open $D \subset \mathbb{R}^2$. Assume that there exists a point $(\bar{x},\bar{y}) \in D$ such that:
	\begin{itemize}
		\item $F(\bar{x},\bar{y})=0$;
		\item $\frac{\partial F}{\partial y}(\bar{x},\bar{y})\neq 0$.
	\end{itemize}
	Then, there exist two positive constant a, b and a function $f : I_{\bar{x}}:=(\bar{x}-a,\bar{x}+a) \mapsto J_{\bar{y}}:=(\bar{y}-b,\bar{y}+b)$ such that
	\[F(x,y)=0 \iff y = f(x) \quad \forall x \in I_{\bar{x}}, \ \forall y \in J_{\bar{y}}\]
	Moreover, $f \in \mathcal{C}^1(I_{\bar{x}})$ and
	\begin{equation}\label{eq:DiniDer}
	f'(x) = -\frac{F_x(x,f(x))}{F_y(x,f(x))} \quad \forall x \in I_{\bar{x}}
	\end{equation}
	where $F_x(x,f(x)) = \frac{\partial F}{\partial x}(x,f(x))$.
\end{thm}

Consider now~\eqref{3rd-degree-eq-scalar-general}, which we rewrite as:
\begin{equation}\label{eq:3rdDegree}
	\rho\delayCont_{\textit{{\textit{opt}}}}^3 + a^2\delayCont_{\textit{{\textit{opt}}}}^2-\frac{1}{4} = 0
\end{equation}
We can see the left-hand term in the previous equation as a parametric function of two positive-valued variables, namely
\begin{equation}
 F: \mathbb{R}_+ \times \mathbb{R}_+ \rightarrow \mathbb{R}, \ (\pi,\delayCont) \mapsto F(\pi,\delayCont) = \rho\delayCont^3 + a^2\delayCont^2-\frac{1}{4}
\end{equation}
where $\pi$, which is either $\rho $ or $a^2$, is a variable, and the other coefficient is a parameter. 
Given a solution $(\bar{\pi}, \bar{\delayCont}_\textit{opt})$ of~\eqref{eq:3rdDegree}, it holds:
\begin{itemize}[leftmargin=*]
	\item $F(\bar{\pi}, \bar{\delayCont}_\textit{opt}) = 0$, by construction;
	\item $F_{\delayCont}(\bar{\rho}, \bar{\delayCont}_\textit{opt}) = 3\bar{\rho}\bar{\delayCont}_\textit{opt}^2 +2{a}^2\bar{\delayCont}_\textit{opt} > 0$, as $ \bar{\rho}, \delayCont_\textit{opt} > 0$;
	\item $F_{\delayCont}(\bar{a^2}, \bar{\delayCont}_\textit{opt}) = 3{\rho}\bar{\delayCont}_\textit{opt}^2 +2\bar{a^2}\bar{\delayCont}_\textit{opt} > 0$, as $ \rho, \delayCont_\textit{opt} > 0$.
\end{itemize}
Then~\cref{thm:Dini} applies and there exists a function $\delayCont(\pi)$ such that $F(\pi,\delayCont_\textit{opt}) = 0 \iff \delayCont_\textit{opt} = \delayCont(\pi)$, with $\pi$ in some open neighbourhood of $\bar{\pi}$. 
Since we did not pose constraints on $ \bar{\pi} $, such a function is defined on the positive real line.
We can then compute the first derivative of $\delayCont(\pi)$ according to~\eqref{eq:DiniDer}.
\begin{description}[leftmargin=*]
	\item[$\boldsymbol{\pi = \rho}$] The first derivative of $\delayCont(\pi) = \delayCont(\rho)$ is
	\begin{equation}
	\delayCont'(\rho) = -\dfrac{F_{\rho}(\rho,\delayCont(\rho))}{F_{\delayCont}(\rho,\delayCont(\rho))} = -\dfrac{{\delayCont(\rho)}^2}{3{\rho}{\delayCont(\rho)} +2a^2} < 0
	\end{equation}
	
	\item[$\boldsymbol{\pi = a^2}$] The first derivative of $\delayCont(\pi) = \delayCont(a^2)$ is
	\begin{equation}
	\delayCont'(a^2) = -\dfrac{F_{a^2}(a^2,\delayCont(a^2))}{F_{\tau}(a^2,\delayCont(a^2))} = -\dfrac{\delayCont(a^2)}{3\rho{\delayCont(a^2)} +2a^2} < 0
	\end{equation}
	
\end{description}
Hence, $\delayCont_\textit{opt}$ is strictly decreasing with both $ \rho $ and $a^2$.
\revision{
\section{Sensor Fusion with Kalman Filter in Information Form and Packet Loss}\label{app:pktLoss}

In the following, we drop the dependencies on preprocessing delays for the sake of exposition.
According to~\cite{schenatoKalmanFusion}, when the correct reception of a measurement from $ i $-th sensor is a binary random variable with success probability $ \lambda_i $, the optimal steady-state estimator with constant gains has the following dynamics for the expected error covariance:
\begin{equation}\label{eq:AREPktLoss}
P = APA^T + Q - APC^T_\lambda\left(C_\lambda PC^T_\lambda+P_\lambda + R_\lambda\right)^{-1}C_\lambda PA^T
\end{equation}
with 
\begin{equation}\label{eq:AREPktLossParams}
	\begin{aligned}
		C_\lambda &= \left[ \lambda_1C_1^T
		\dots
		\lambda_{\tilde{\cardS}}C_{\tilde{\cardS}}^T \right]^T\\
		P_\lambda &= \mbox{diag}\left(\lambda_1(1-\lambda_1)C_1PC_1^T,...,\lambda_{\tilde{\cardS}}\left(1-\lambda_{\tilde{\cardS}}\right)C_{\tilde{\cardS}}PC_{\tilde{\cardS}}^T\right)\\
		R_\lambda &= \mbox{diag}\left(\lambda_1R_1,...,\lambda_{\tilde{\cardS}}R_{\tilde{\cardS}}\right)
	\end{aligned}
\end{equation}
where $ \mbox{diag}(\cdot) $ denotes a block-diagonal matrix with variables as diagonal blocks,
and $ \{1,...\tilde{\cardS}\} $ are the sensors involved in the measurement update.
\cref{eq:AREPktLoss} can be rewritten as follows, exploiting the matrix inversion lemma:
\begin{equation}\label{eq:AREPktLossInfoForm}
P = A\left(P^{-1} + C_\lambda^T\left(P_\lambda+R_\lambda\right)^{-1}C_\lambda\right)^{-1}A^T + Q
\end{equation}
where we define the modified information matrix as 
\begin{equation}\label{eq:infoMatrixPktLoss1}
\tilde{\Gamma} = C_\lambda^T\left(P_\lambda+R_\lambda\right)^{-1}C_\lambda
\end{equation}
We then get~\cref{eq:infoMatrixPktLoss2}, where $ (i) $ follows from the matrix inversion lemma, and $ (ii) $ from the definition of $ \Gamma_i $.}
\begin{figure*}[b]
	\centering
	\noindent\makebox[\linewidth]{\rule{\linewidth}{0.4pt}}
	\revision{\begin{equation}\label{eq:infoMatrixPktLoss2}
		\begin{aligned}
		\tilde{\Gamma} &= \sum_{i\in\mathcal{S}(\delayDisc{})}
			\lambda_i^2C_i^T\left(\lambda_i(1-\lambda_i)C_iPC_i^T + \lambda_iR_i\right)^{-1}C_i = \\
		&= \sum_{i\in\mathcal{S}(\delayDisc{})}
			\lambda_i^2C_i^T\left[\lambda_i(1-\lambda_i)\left(C_iPC_i^T + \dfrac{R_i}{1-\lambda_i}\right)\right]^{-1}C_i =\\ 
		&\stackrel{(i)}{=} \sum_{i\in\mathcal{S}(\delayDisc{})}
			\lambda_iC_i^T\left[R_i^{-1}-(1-\lambda_i)R_i^{-1}C_i\left(P^{-1}+\left(1-\lambda_i\right)C_i^TR_i^{-1}C_i\right)^{-1}C^T_iR^{-1}_i\right]C_i=\\
		&=\sum_{i\in\mathcal{S}(\delayDisc{})}
			\lambda_i\left[C_i^TR_i^{-1}C_i-\left(1-\lambda_i\right)C_i^TR_i^{-1}C_i\left(P^{-1}+\left(1-\lambda_i\right)C_i^TR_i^{-1}C_i\right)^{-1}C_i^TR_i^{-1}C_i\right]=\\
		&\stackrel{(ii)}{=}\sum_{i\in\mathcal{S}(\delayDisc{})}
			\lambda_i\left[\Gamma_i-\left(1-\lambda_i\right)\Gamma_i\left(P^{-1}+\left(1-\lambda_i\right)\Gamma_i\right)^{-1}\Gamma_i\right]
		\end{aligned}
	\end{equation}}
\end{figure*}


\section{Proof of~\cref{thm:costFunction}}\label{app:proofThmCostFunction}

According to~\cite[Section 3]{schenato2008}, the estimation starts from the most recent state for which the maximum information possible is available.
The former has timestamp $ k-\delayDiscTot{\cardS}-\delayFusDiscTot $, being $ \delayDiscTot{\cardS} $ the delay gathered by the most-delayed-sensor data when they are received at the central station.
The \revision{expected} error covariance for such estimate converges to the solution of the ARE~\eqref{eq:ARE} where all sensors are considered, that is, at steady state the following holds:
\begin{multline}\label{eq:AREApp}
P_{k-\delayFusDiscTot-\delayDiscTot{\cardS}|k-1-\delayFusDiscTot-\delayDiscTot{\cardS}}(\delaySetDisc{}{}) = \\
P_{k-\delayFusDiscTot-\delayDiscTot{\cardS}+1|k-\delayFusDiscTot-\delayDiscTot{\cardS}}(\delaySetDisc{}{}) = P_\infty(\delaySetDisc{}{})
\end{multline}
When computing the state estimates of more recent times, only data from some sensors are available for fusion.
In particular, the measurement update for the estimate of the state with delay $ \delay + \delayFusDiscTot $ can only use sensors in $ \mathcal{S}(\delay) $:
\begin{equation}
P_{k-\delayFusDiscTot-\delay|k-\delayFusDiscTot-\delay}\left(\delaySetDisc{}{}\right)  = \update\left(P_{k-\delayFusDiscTot-\delay|k-\delayFusDiscTot-\delay-1}\left( \delaySetDisc{}{}\right),\delaySetDisc{\delay}{}\right)
\end{equation}
According to~\cref{ass:sensorSorting}, the multi-step KF iteration processing data in the interval $ \left[ k-\delayFusDiscTot-\delayDiscTot{i+1}+2, \, k-\delayFusDiscTot-\delayDiscTot{i}+1 \right] $ involves the sensor subset $ \mathcal{S}(\delayDiscTot{i}) = \{1,...,i\} $.
The resulting revision{expected} error covariance for such iteration is, according to~\eqref{eq:KalmanMultiUpdate},
\begin{multline}\label{eq:multiStepIter}
		P_{k-\delayFusDiscTot-\delayDiscTot{i}+1|k-\delayFusDiscTot-\delayDiscTot{i}}(\delaySetDisc{}{}) = \\
		= \iteration^{\delayDiscTot{i+1}-\delayDiscTot{i}}\left( P_{k-\delayFusDiscTot-\delayDiscTot{i+1}+1|k-\delayFusDiscTot-\delayDiscTot{i+1}}\left(\delaySetDisc{}{}\right),\delaySetDisc{\delayDiscTot{i}}{}\right)
\end{multline}
The multi-step KF iteration involving all the processed dataset~\eqref{def:datasets} is written as
$ \iteration^{\delayDiscTot{\cardS}-\delayDiscTot{1}}\left(P_{\infty}\left(\delaySetDisc{}{}\right),\delaySetDisc{\delayDiscTot{\cardS}-1}{}\right) $: starting from $ P_{\infty}\left(\delaySetDisc{}{}\right) $, it computes $ P_{k-\delayFusDiscTot-\delayDiscTot{1}+1|k-\delayFusDiscTot-\delayDiscTot{1}}(\delaySetDisc{}{})  $ through the multi-step KF iterations~\eqref{eq:multiStepIter}, each involving one sensor less than the previous one.
The multi-step prediction $ \timeUp^{\delayDisc{\textit{pred}}}(\cdot) $ eventually computes the estimate of the current state, where the remaining delay is $ \delayDisc{\textit{pred}}=\delayFusDiscTot+\delayDiscTot{1}-1 $.

\end{document}